\RequirePackage{fix-cm}
\documentclass[smallextended]{svjour3}
\journalname{Journal of Scientific Computing}
\smartqed  
\usepackage{amsmath}
\usepackage{amsthm}
\usepackage{graphicx}
\usepackage{amsopn}
\usepackage{lipsum}
\usepackage{amsfonts}
\usepackage{graphicx}
\usepackage{epstopdf}
\usepackage{algorithm,algorithmic}
\usepackage{siunitx}
\usepackage{booktabs}
\usepackage{amssymb}
\usepackage{multirow}
\usepackage{makecell}
\usepackage{hyperref} 
\usepackage{caption}
\usepackage{subcaption}

\begin{document}
\begin{sloppypar}
\title{Proximal Gradient Descent Ascent Methods for Nonsmooth Nonconvex-Concave Minimax Problems on Riemannian Manifolds
}

\titlerunning{PGDA for Nonsmooth NC-C Minimax Problems
on Riemannian Manifolds}        

\author{Xiyuan Xie \and Qia Li        
}

\institute{Xiyuan Xie \at
              School of Mathematics, Sun Yat-sen University, Guangzhou 510275, China \\
              \email{xiexy65@mail2.sysu.edu.cn}
           \and
           Qia Li \at
              School of Computer Science and Engineering, Guangdong Province Key Laboratory of Computational Science, Sun Yat-sen University, Guangzhou 510275, China \\
              Corresponding author. \\
              \email{liqia@mail.sysu.edu.cn}
}

\date{Received: date / Accepted: date}

\maketitle

\begin{abstract}
Nonsmooth nonconvex-concave minimax problems have attracted significant attention due to their wide applications in many fields. In this paper, we consider a class of nonsmooth nonconvex-concave minimax problems on Riemannian manifolds. Owing to the nonsmoothness of the objective function, existing minimax manifold optimization methods cannot be directly applied to solve this problem. We propose two manifold proximal gradient descent ascent (MPGDA) algorithms for solving the problem. The first  algorithm alternatively performs one or multiple manifold proximal gradient descent steps and a proximal ascent step at each iteration, and we prove that it can find an $\varepsilon$-game-stationary point and an $\varepsilon$-optimization-stationary point  within $\mathcal{O}(\varepsilon^{-3})$ outer iterations. The second algorithm alternatively performs one manifold proximal gradient descent step and a proximal gradient ascent step, and we show that it can reach an $\varepsilon$-game-stationary point and an $\varepsilon$-optimization-stationary point within $\mathcal{O}(\varepsilon^{-4})$ outer iterations. Numerical experiments on an analytic example,  fair sparse PCA, and sparse spectral clustering are conducted to illustrate the advantages of the proposed algorithms.

\end{abstract}

\section{Introduction}
\label{intro}
In this paper, we consider a class of nonsmooth  nonconvex-concave minimax optimization problems formulated as follows:
\begin{align}\label{Prob}
		\min_{x\in \mathcal{M}}\max_{y\in S}F(x, y):=f(x, y)+h(x)-g(y), 
\end{align}
where $\mathcal{M}$ is a Riemannian manifold embedded in a finite-dimensional Euclidean space $\mathbb{E}_1$,  and $S$ is a convex set in another finite-dimensional Euclidean space $\mathbb{E}_2$, $f:\mathbb{E}_1\times \mathbb{E}_2\to \mathbb{R}$ is a continuously differentiable function and is concave with respect to y,  $h:\mathbb{E}_1\to \mathbb{R}$ and $g:\mathbb{E}_2\to \mathbb{R}$ are some proper closed convex (possibly nonsmooth) functions. Throughout this paper, we adopt the following assumptions for problem \eqref{Prob}.

\begin{assumption}\label{Assumption}
\par
~~

        \textbf{$(i)$} The manifold $\mathcal{M}$ and the convex set $S$ are compact. Let $\sigma_y:=\sup_{y\in S}\left\|y\right\|<\infty$.       
        
        \textbf{$(ii)$} The function $f$ has Lipschitz continuous gradient with respect to $x$, i.e., there exists a Lipschitz constant $L_x>0$ such that  \[\left\|\nabla_x f(x_1, y)-\nabla_x f(x_2, y)\right\|\leq L_x\left\|x_1-x_2\right\|\]
    for any $x_1,x_2\in \mathcal{M}$ and $y\in S$.

        \textbf{$(iii)$} The function $\nabla_y f(x,y)$ is  Lipschitz continuous with respect to $(x,y)$, i.e., there exists a Lipschitz constant  $L_y>0$ such that \[\left\|\nabla_y f(x_1, y_1)-\nabla_y f(x_2, y_2)\right\|^2\leq L_y^2(\left\|x_1-x_2\right\|^2+\left\|y_1-y_2\right\|^2) \]
    for any $(x_1, y_1)$, $(x_2, y_2)\in \mathcal{M}\times S$.
    
        \textbf{$(iv)$} The function $h$ is $L_h$ Lipschitz continuous.

   \textbf{$(v)$} For all $u\in\mathbb{E}_1,~w\in\mathbb{E}_2$ and $\alpha>0$, the following two proximal problems can be solved exactly and efficiently: 
   \begin{subequations}
       \begin{align}
           \min_{x\in\mathbb{E}_1} h(x)+\frac{1}{2\alpha}\|x-u\|^2,\label{proximal-h}\\
           \min_{y\in S} g(y)+\frac{1}{2\alpha}\|y-w\|^2.\label{proximal-g}
       \end{align}
   \end{subequations}
\end{assumption}

Problem \eqref{Prob} has many important applications in machine learning and signal processing. Note that if $f(x,y)$ is linear with respect to $y$ for fixed $x$, then $f(x,y) = \langle \mathcal{A}(x), y \rangle$ for some smooth operator $\mathcal{A}:\mathbb{E}_1\to\mathbb{E}_2$. Below, we present two representative examples and refer the interested readers to \cite{chenProximalGradientMethod2020,BriefIntroductiontoManifoldOptimization} for more examples.

\begin{example}
   \textbf{Fair Sparse Principal Component Analysis (FSPCA)}. The classical PCA \cite{Hotelling1933Analysis} is one of the most widely used dimensionality reduction techniques, which minimizes the total reconstruction error over the whole samples. However, in real-life applications, the data samples come from diverse classes and PCA  may exhibit a higher reconstruction error for certain classes than the others. To address this disparity, the fair PCA is  proposed \cite{samadiPriceFairPCA}. 
Suppose the $m$ data samples belong to $n$ classes, and each class $i$ corresponds to $A_i\in\mathbb{R}^{m_i\times d}$ with $d$ being the dimension of the samples, $m_i$ being the number of samples in the $i$-th class, and $\sum_{i=1}^n m_i = m$. Then the fair  PCA can be formulated as:
\begin{align} \min_{X\in \mathrm{St}(d, r)}\max_{i=1, \dots, n}-\mathrm{Tr}(X^\top A_i^\top A_iX), 
\end{align}
where  $r<d$ is the number of principal components, $\mathrm{Tr}(\cdot)$ represents the trace and $\mathrm{St}(d, r):=\{X\in\mathbb{R}^{d\times r}|X^\top X=I_r\}$ with $I_r$ denoting the $r\times r$ identity matrix. Recently, the fair sparse PCA is further proposed to promote sparsity in the principal components \cite{babuFairPrincipalComponent2023} as follows:
\begin{align} \label{SFP1} \min_{X\in \mathrm{St}(d, r)}\max_{i=1, \dots, n}-\mathrm{Tr}(X^\top A_i^\top A_iX)+\mu \left\|X\right\|_1, 
\end{align} 
where $\mu>0$ is a weighting parameter and 
$\left\|X\right\|_1:=\sum_{i,j}\left|x_{ij}\right|$ is the $\ell_1$-norm of the matrix $X$. In view of \cite[Example~4.10]{Beck2017First}, problem \eqref{SFP1} can be equivalently reformulated into 
\begin{align}\label{FSPCA}
    \min_{X\in \mathrm{St}(d, r)}\max_{y\in\Delta_n}-\sum_{i=1}^n y_i\mathrm{Tr}(X^\top A_i^\top A_iX)+\mu \left\|X\right\|_1.
\end{align}
where $\Delta_n:=\{y|\sum_{i=1}^n y_i=1, y_i\geq0, i=1, \dots, n\}$ is the standard simplex in $\mathbb{R}^n$. Note that problem \eqref{FSPCA} is an instance of  problem \eqref{Prob} with $\mathcal{M}$ being the Stiefel manifold $\mathrm{St}(d,r)$, $S$ being the standard simplex $\Delta_n$, $h(X)=\mu\|X\|_1$ and $f(X,y)=-\sum_{i=1}^n y_i\mathrm{Tr}(X^\top A_i^\top A_iX)$.
\end{example}

\begin{example}\textbf{Sparse Spectral Clustering (SSC)}. The SSC problem aims to partition $N$ data samples into 
$p$ groups such that similar data points are clustered together. In spectral clustering, a symmetric affinity matrix $W = [W_{ij}]_{N\times N}$ is constructed, where $W_{ij} \geq 0$
 quantifies the pairwise similarity between samples $a_i$ and $a_j$. To promote sparsity and interpretability, the sparse spectral clustering  framework is introduced in \cite{2016ConvexSparseSpectralClustering},  as follows:
\begin{align}\label{SSC1}
    \min_{X\in\mathrm{St}(N,p)}\langle L, XX^{\top}\rangle+\mu\|XX^{\top}\|_1,
\end{align}
where $\mu > 0$ is the regularization parameter and $L = I_N -S^{-1/2}WS^{-1/2}$ is the normalized Laplacian matrix with
$S^{1/2}$ being the diagonal matrix with diagonal elements $\sqrt{s_1}, \sqrt{s_2}, \dots, \sqrt{s_n}$
and $s_i = \sum_{j} W_{ij}$. It is not hard to see that problem \eqref{SSC1} can be equivalently rewritten as:
\begin{align}\label{SSC2}
    \min_{X\in\mathrm{St}(N,p),~Z\in\mathbb{R}^{N\times N}}\langle L, XX^{\top}\rangle+\mu\|Z\|_1+\mu\|XX^{\top}-Z\|_1.
\end{align}
By the dual representation of the $\ell_1$-norm, i.e., $\|\cdot\|_1=\max_{\|Y\|_\infty \leq1}\langle Y,\cdot\rangle$,  problem \eqref{SSC2} can be further equivalently reformulated as:  
\begin{align}\label{SSC3}
    \min_{X\in\mathrm{St}(N,p),~Z\in\mathbb{R}^{N\times N}}\max_{\|Y\|_\infty \leq \mu}\langle L, XX^{\top}\rangle+\mu\|Z\|_1+\langle Y, XX^{\top}-Z\rangle.
\end{align}
Again, problem \eqref{SSC3} is an instance of problem \eqref{Prob} with $\mathcal{M}$ being the product manifold $\mathrm{St}(N,p)\times \mathbb{R}^{N\times N}$, $S=\{Y\in\mathbb{R}^{N\times N}\mid \|Y\|_{\infty}\leq \mu\}$, $h(X,Z)=\mu \|Z\|_1$ and $f(X,Z,Y)=\langle L,XX^{\top}\rangle+\langle Y, XX^{\top}-Z\rangle$.
\end{example}

 In recent years, nonconvex-concave (NC-C) minimax problems have received a tremendous amount of attention as they have found broad applications in diverse areas. Many studies focus on NC-C problems in the setting where both the constraints on $x$ and $y$ are closed convex sets and numerous methods have been developed, {see, e.g., \cite{li2026smoothing,linGradientDescentAscent,linOptimalAlgorithmsMinimax,nouiehedSolvingClassNonConvex,ostrovskiiEfficientSearchFirstOrder2021,panEfficientAlgorithmNonconvexlinear2021a,2022Zeroth,thekumparampilEfficientAlgorithmsSmooth,xuUnifiedSingleloopAlternating2023,yangFasterSingleloopAlgorithms,zhangSingleLoopSmoothedGradient}.} In contrast, there are relatively few works studying NC-C minimax problems in the setting where the constraint on $x$ is a Riemannian  submanifold and the constraint on $y$ is a closed convex set. Among these works, Huang and Gao \cite{huangGradientDescentAscent2023} propose a Riemannian gradient descent-ascent algorithm for a smooth nonconvex-strongly concave (NC-SC) minimax problem and establish its iteration complexity of $\mathcal{O}(\varepsilon^{-2})$ for achieving $\varepsilon$-stationarity. Later, Xu et al. \cite{xuEfficientAlternatingRiemannian2024} present an alternating Riemannian/projected gradient descent-ascent (ARPGDA) algorithm with an iteration complexity of $\mathcal{O}(\varepsilon^{-3})$ for finding an $\varepsilon$-stationary point of smooth Riemannian nonconvex-linear (NC-L) minimax problem. Recently, for problem \eqref{Prob} with additional assumptions that $h$ is smooth (but not necessarily convex) and $f(x,y)$ is linear with respect to $y$ for fixed $x$, the work \cite{xuRiemannianAlternatingDescent2024} extends the technique in \cite{heApproximationProximalGradient2024} and develops algorithms that can reach $\varepsilon$-stationary points within $\mathcal{O}(\varepsilon^{-3})$ iterations. {Very recently, Aybat, Hu, and Deng \cite{aybat2025retraction} propose a retraction-free smoothed manifold gradient descent-ascent method to tackle the case where $h=0$ and $f(x,y)-g(y)$ is concave with respect to $y$ for fixed $x$. Their algorithm can reach an $\varepsilon$-stationary point within \(\mathcal{O}(\varepsilon^{-4})\) iterations, which can be improved to \(\mathcal{O}(\varepsilon^{-3})\) by incorporating a Tikhonov regularization under suitable initialization conditions (see Remark 11 of \cite{aybat2025retraction}).}
 It is worth noting that none of the methods in the aforementioned literature \cite{huangGradientDescentAscent2023,xuRiemannianAlternatingDescent2024,xuEfficientAlternatingRiemannian2024,aybat2025retraction} can directly tackle problem \eqref{Prob} due to the following two reasons: (i) the convex function $h$ may be nonsmooth, and (ii) the function $f(x, y)$ is neither linear nor strongly concave with respect to $y$ for fixed $x$. Based on the needs of applications and the above background, it is highly desirable to develop efficient methods with theoretical convergence guarantees for solving the Riemannian nonsmooth NC-C minimax problem \eqref{Prob}.

The main contributions of this paper are summarized below. First, we introduce the notions of  optimization stationarity as well as  game stationarity for problem \eqref{Prob} and carefully study their relationships. Second, we propose two manifold proximal gradient descent ascent (MPGDA) algorithms for solving problem \eqref{Prob}. The first algorithm, denoted by MPGDA-PA, performs one or multiple manifold proximal gradient descent steps to update $x$ and one proximal ascent step to update $y$ at each iteration. Note that the proximal ascent subproblem with respect to $y$ can generally be solved by FISTA \cite[Theorem~10.41]{Beck2017First} with a linear convergence rate. In particular, for the case of $f$ being linear with respect to $y$, addressing this subproblem reduces to solving a proximal problem in the form of  \eqref{proximal-g}. The second algorithm, denoted by MPGDA-PGA, also performs one manifold proximal gradient descent step to update $x$ but uses a single proximal gradient ascent step to renew $y$ at each iteration. {It is worth noting that the $y$-update in MPGDA-PGA requires only a single evaluation of the gradient $\nabla_y f$ and one step of solving  the proximal problem \eqref{proximal-g}.}  This has much lower computational cost than that of MPGDA-PA  in the case where $f$ is not linear with respect to $y$. Third, we prove that the MPGDA-PA algorithm can find an $\varepsilon$-game-stationary point of problem (1) within $\mathcal{O}(\varepsilon^{-3})$ outer iterations, while the MPGDA-PGA algorithm achieves this within $\mathcal{O}(\varepsilon^{-4})$ outer iterations. Finally, we illustrate the proposed algorithms via both analytic and simulated numerical examples.

The rest of the paper is organized as follows. Section \ref{sec:2} presents preliminaries that will be used in this paper. Section \ref{sec:3} discusses the optimality condition of problem \eqref{Prob}. Section \ref{sec:4} introduces the proposed MPGDA-PA method and establishes its iteration complexity. Section \ref{sec:5} introduces the proposed MPGDA-PGA method and establishes its iteration complexity. Section \ref{sec:6} reports the numerical experiments.

\section{Preliminaries}\label{sec:2}
 In this section, we recall some basic notations and preliminary results which will be used in this paper. Let 
 $\mathbb{E}$ represent a finite-dimensional Euclidean space with inner product $\langle \cdot, \cdot\rangle$ and the induced norm $\left\|\cdot\right\|$. For a linear operator $\mathcal{L}:\mathbb{E}_1\to \mathbb{E}_2 $, let $\mathrm{rank}(\mathcal{L})$ denote the dimension of its image space $\{y\in\mathbb{E}_2\mid y=\mathcal{L}(x),~x\in \mathbb{E}_1\}$ and $\mathrm{ker}(\mathcal{L})$ denote its null space $\{x\in\mathbb{E}_1\mid \mathcal{L}(x)=0\}$ . Given a point $x\in\mathbb{E}$ and a set $C\subseteq \mathbb{E}$, let  $\mathrm{dist}(x,C):=\inf_{y\in C}\|y-x\|$ and $\delta_C$ denote the indicator function associated with $C$. For a function $\varphi(x,y):\mathbb{E}_1\times\mathbb{E}_2\to \mathbb{R}$, let  $\nabla_x \varphi(x, y)$ and $\nabla_y \varphi(x, y)$ respectively denote the partial gradients of $\varphi(x, y)$ with respect to $x$ and $y$ at the point $(x,y)$. Let $\mathcal{M}$ represent a Riemannian submanifold $\mathbb{E}_1$. For a function $\Phi:\mathcal{M}\to\mathbb{R}$, $\mathrm{grad} \Phi(x)$ denotes the Riemannian gradient of $\Phi$. For a function $\varphi:\mathcal{M}\times\mathbb{E}_2\to \mathbb{R}$, $\mathrm{grad}\varphi(\cdot,y)$ denotes the Riemannian gradient of $\varphi(\cdot,y)$ with $y$ fixed. Let $\overline{\mathbb{R}}:=\mathbb{R}\{\pm\infty\}$.
 
\subsection{Riemannian submanifold} 
We first review the definition of an embedded submanifold in a Euclidean space $\mathbb{E}$. We say that function $\psi:U\subset\mathbb{E}\to\mathbb{R}^k$ is smooth if it is infinitely differentiable on $U$. Let $\mathrm{D}\psi(x)$ denote its differential at $x$.

\begin{definition}\cite[Definition~3.10]{boumal2023intromanifolds}
Let $\mathbb{E}$ be a Euclidean space of dimension $d$. A nonempty subset $\mathcal{M}$ of $\mathbb{E}$ is called an embedded submanifold of $\mathbb{E}$ of dimension $ n $ if either $ n = d $ and $ \mathcal{M} $ is open in $ \mathbb{E}$  or
     $ n = d - k $ for some $ k \geq 1 $ and, for each $ x \in \mathcal{M} $, there exists a neighborhood $ U $ of $ x $ in $ \mathbb{E} $ and a smooth function $ \psi: U \to \mathbb{R}^k $ such that:
    \begin{enumerate}
        \item If $ y $ is in $ U $, then $ \psi(y) = 0 $ if and only if $ y \in \mathcal{M} $; and
        \item $ \mathrm{rank}(\mathrm{D}\psi(x)) = k $.
    \end{enumerate}
Such a function $ \psi $ is called a local defining function for $ \mathcal{M}$ at $ x $.
\end{definition}
Let $\mathcal{M}$ be an embedded submanifold of $\mathbb{E}$. The tangent space $T_x\mathcal{M}$ to $\mathcal{M}$ at $x\in\mathcal{M}$ is identified with $\mathbb{E}$ when $\mathcal{M}$ is open in $\mathbb{E}$;  otherwise,  $T_x\mathcal{M}=\mathrm{ker}(\mathrm{D}\psi(x))$ , where $\psi$ is a  local defining function for $\mathcal{M}$ at $x$ \cite[Theorem~3.15]{boumal2023intromanifolds}. The embedded submanifold $\mathcal{M}$ is called a Riemannian submanifold of $\mathbb{E}$, when equipped with the Riemannian metric induced by the ambient space $\mathbb{E}$. Namely, for each $x\in \mathcal{M}$, the inner product on $T_x\mathcal{M}$ is defined as $\langle u,v\rangle_x:=\langle u,v\rangle$ for $u,~v\in T_x\mathcal{M}$. The Riemannian gradient of a smooth function $\varphi:\mathbb{E}\to\mathbb{R}$ at a point $x\in\mathcal{M}$ is given by $\mathrm{grad} \varphi(x):=\mathrm{Proj}_{T_{x}\mathcal{M}}(\nabla \varphi(x))$, where $\mathrm{Proj}_{T_{x}\mathcal{M}}(\cdot)$ is the Euclidean projection operator onto $T_x\mathcal{M}$.  The normal space $N_x\mathcal{M}$ of a Riemannian submanifold $\mathcal{M}$ of $\mathbb{E}$ is the orthogonal complement of the tangent space $T_x\mathcal{M}$ in $\mathbb{E}$ \cite[Definition~5.48]{boumal2023intromanifolds}. 

The tangent bundle of $\mathcal{M}$ is the disjoint union of its tangent spaces : $T\mathcal{M}:=\{(x,v)\mid x\in\mathcal{M},~v\in T_x\mathcal{M}\}$.  A retraction on a manifold $\mathcal{M}$ is a smooth mapping $R$ from the tangent bundle $T\mathcal{M}$ onto $\mathcal{M}$ such that, for any $(x,v)\in T_x\mathcal{M} $,  its restriction $R_x:T_x\mathcal{M}\to \mathcal{M}$  satisfies (i) $R_x(0) = x$ and (ii) $\mathrm{D}R_x(0)[v]=v$ \cite[Definition~3.47]{boumal2023intromanifolds}.

The concept of a retraction offers a practical link between manifold and its tangent bundle. The following properties of retractions are needed for convergence analysis in this paper.
\begin{proposition}\cite[Appendix~B]{boumalGlobalRatesConvergence2019}\label{retr} Let $\mathcal{M}$ be a compact Riemannian submanifold of  Euclidean space $\mathbb{E}$ equipped with norm $\left\|\cdot\right\|$. $R:T\mathcal{M} \to \mathcal{M}$ is a retraction. Then there exist constants $M_1>0$ and $M_2>0$ such that the following two inequalities hold for any $x\in \mathcal{M}$ and $v\in T_x{\mathcal{M}}$:
    \begin{align*}
        \left\|R_x(v)-x\right\|&\leq M_1\left\|v\right\|, \\
        \left\|R_x(v)-(x+v)\right\|&\leq M_2\left\|v\right\|^2. 
    \end{align*}
\end{proposition}

\subsection{Generalized subdifferential}\label{sec-general-diff}  
For a function $\varphi:\mathbb{E}\to \mathbb{R}$, the usual one-sided directional derivative of $\varphi$ at $x\in\mathbb{E}$ along the direction $v\in\mathbb{E}$ is defined by
\[\varphi'(x;v):=\lim_{t\downarrow0}\frac{\varphi(x+tv)-\varphi(x)}{t} .\]  In contrast to the usual one-sided directional derivative, the (Clarke) generalized  directional derivative  does not presuppose the existence of  any limit \cite[Section~2.1]{1983Optimization}. Specifically,  the (Clarke) generalized  directional derivative of $\varphi$ at $x\in\mathbb{E}$ along the direction $v\in\mathbb{E}$ is defined as  
\[\varphi^\circ(x;v):=\lim\sup_{\begin{smallmatrix}y\to x\\t\downarrow0\end{smallmatrix}}\frac{\varphi(y+tv)-\varphi(y)}{t}, \]
and the (Clarke) generalized  subdifferential of $\varphi$ at $x$ is defined as
\[\partial \varphi(x):=\{z\in\mathbb{E}\mid \langle z, v\rangle\leq \varphi^\circ(x;v),~\forall v\in\mathbb{E}\}. \] 
If $\varphi:\mathbb{E}\to \mathbb{R}$ is convex, then for any $x\in\mathbb{E}$, $\partial f(x)$ coincides with the subdifferential at $x$ in the sense of convex analysis.

A function $\varphi:\mathbb{E}\to\mathbb{R}$ is said to be regular at $x\in\mathbb{E}$, provided that for all $v\in \mathbb{E}$, the usual one-sided directional derivative $\varphi'(x;v)$ exists and $\varphi'(x;v)=\varphi^\circ(x;v)$  \cite[Definition~2.3.4]{1983Optimization}. If $\varphi_1:\mathbb{E}\to \mathbb{R}$ and $\varphi_2:\mathbb{E}\to \mathbb{R}$ are locally Lipschitz continuous and regular at $x\in\mathbb{E}$, then $\varphi_1+\varphi_2$ are regular at $x$ {\cite[Proposition~2.3.6 (c)]{1983Optimization},} and $\partial (\varphi_1+\varphi_2)(x)=\partial \varphi_1(x)+\partial \varphi_2(x)$ {\cite[Corollary~3 of Proposition~2.3.3]{1983Optimization}.}  If $\varphi:\mathbb{E}\to \mathbb{R}$ is a continuously differentiable function or a real-valued convex function, then it is regualr at any $x\in\mathbb{E}$ {\cite[Proposition~2.3.6 (a) (b)]{1983Optimization}.} A function $\varphi:\mathbb{E}\to\mathbb{R}$ is said to be $\ell$-weakly convex if $\varphi(\cdot)+(\ell/2)\|\cdot\|^2$ is convex. It follows that the real-valued weakly convex function $\varphi$ is regular at any $x\in\mathbb{E}$.

{For a given nonempty closed set $C\subseteq \mathbb{E}$ and $x\in \mathbb{E}$, let $d_{C}(x):=\mathrm{dist}(x,C)$. The tangent cone to $C$ at $x$ is defined by $T_{x}C:=\{v\in\mathbb{E}\mid d^{\circ}_C(x;v)=0\}$, the normal cone is defined by $N_{x}C:=\{z\in\mathbb{E}\mid \langle z,v\rangle\leq0,~\forall~v\in T_{x}C\}$. If $\varphi:\mathbb{E}\to\mathbb{R}$ is locally Lipschitz continuous around $x$ and attains its minimum over $C$ at $x$, then $0\in\partial \varphi(x)+N_C(x)$ \cite[Corollary of Proposition~2.4.3]{1983Optimization}.}

\section{Optimality Conditions}\label{sec:3}
 In this section, we introduce two important concepts of stationarity  for problem \eqref{Prob}, namely, the optimization stationarity and game stationarity, and examine their relationships. To this end, we first define the value function $\Phi:\mathbb{E}_1\to\mathbb{R}$ by  \begin{align}\label{value function}
     \Phi(x):=\max_{y\in S} f(x, y)-g(y).
 \end{align}
Then the minimax problem \eqref{Prob} can be equivalently reformulated into
    \begin{align}\label{Prob2}
         \min_{x\in\mathcal{M}} \Phi(x)+h(x).  
    \end{align}
The next proposition concerns the value function $\Phi$. 
\begin{proposition}\label{pro:value function}
      The value function $\Phi$ defined in \eqref{value function} is $L_x$-weakly convex and for all $x\in\mathbb{E}_1$, it holds that  $\nabla_x f(x, y_x)\in \partial \Phi(x)$ with $y_x\in \mathop{\mathrm{argmax}}_{y\in S} \{f(x, y)-g(y)\}$.  Moreover,  if $f(x, y)-g(y)$ is strongly concave in $y$,  then $\Phi$ is differentiable and  $\nabla \Phi(x)=\nabla_x f(x, y_x)$. 
\end{proposition}
\begin{proof}
    Because $f(x, y)$ is $L_x$-smooth with respect to $x$,  $f(x, y)+\frac{L_x}{2}\left\|x\right\|^2-g(y)$ is convex in $x$ for any given $y\in S$. Thus,  $\Phi(x)+\frac{L_x}{2}\left\|x\right\|^2=\max_{y\in S} f(x, y)+\frac{L_x}{2}\left\|x\right\|^2-g(y)$ is convex. Therefore,  $\Phi(x)$ is $L_x$-weakly convex. Since $S$ is compact, it follows from  Danskin's theorem \cite[Theorem~D1]{bernhardTheoremDanskinApplication1995} that $\nabla_x f(x, y_x)\in \partial\Phi(x)$, and in particular,  $\nabla \Phi(x)=\nabla_x f(x, y_x)$ when $f(x, y)-g(y)$ is strongly concave in $y$.    
\end{proof}   

{Proposition \ref{pro:value function} shows that $\Phi$ is weakly convex and real-valued, and thus it is locally Lipschitz continuous and regular. Consequently, by \cite[Corollary of Proposition~2.4.3]{1983Optimization}, any local minimizer $x^*$ of problem \eqref{Prob2} must satisfy
\begin{align*}
     0\in \partial (\Phi + h)(x^*) + N_{x^*}\mathcal{M}.
\end{align*}
In view of \cite[Corollary~3 of Proposition~2.3.3]{1983Optimization}, and the regularity of $\Phi$ and $h$, this condition is equivalent to \begin{align}\label{eq:optimization-stationary}
    0\in  \partial \Phi(x^*)+\partial h(x^*)+N_{x^*}\mathcal{M}.
\end{align}}


  This naturally leads to the following definition of optimization stationarity, where the $\varepsilon$-optimization-stationarity is partly motivated by \cite{chenProximalGradientMethod2020} and \cite{perez2019subdifferential}.
\begin{definition}[Optimization stationarity]\label{def:optimization-stationary}
Consider problem \eqref{Prob} and the value function $\Phi$ defined in \eqref{value function}.
\begin{enumerate}
    \item[(i)] We say $x^*$ is a  optimization-stationary point of problem \eqref{Prob} if it satisfies \eqref{eq:optimization-stationary}.
    \item[(ii)] Let $\varepsilon>0$ be given. We say that $x^*$ is an $\varepsilon$-optimization-stationary point of problem \eqref{Prob} if there exists $u\in\mathbb{E}$ such that 
{\[\max\{\|u\|,\mathrm{dist}(0,\partial_{\varepsilon}\Phi(x^*+u)+\partial h(x^*+u)+N_{x^*}\mathcal{M})\}<\varepsilon,\]}
where \[\partial_{\varepsilon}\Phi(x):=\{z\in\mathbb{E}\mid \Phi(x)+ \langle z, y-x\rangle-\varepsilon\leq \Phi(y)+\frac{L_x}{2}\|y-x\|^2,~\forall y\in\mathbb{E}\}. \]

\end{enumerate}     
\end{definition}

Now we consider defining the game stationarity of problem \eqref{Prob}. It is routine to prove that any (local) saddle point $(x^*,y^*)$ of problem \eqref{Prob} satisfies 
{\begin{subequations}
    \begin{align}
	&0\in \nabla_x f(x^*,y^*)+\partial h(x^*)+N_{x^*}\mathcal{M}, \label{eq1:game-stationary}\\  &0\in \nabla_y f(x^*, y^*)-\partial g(y^*)-N_{y^*}S, \label{eq2:game-stationary}
   \end{align}  
\end{subequations}}
which yields the following definition of a game-stationary point of problem \eqref{Prob}.
\begin{definition}[Game-stationary point]
    We say that $(x^*, y^*)$ is a game-stationary point of problem \eqref{Prob} if it satisfies \eqref{eq1:game-stationary} and \eqref{eq2:game-stationary}. 
\end{definition}

Next, we study the relationships between optimization-stationary points and game-stationary points of problem \eqref{Prob}.

\begin{proposition}
     If $x^*$ is an optimization-stationary point of the  problem \eqref{Prob}, then there  exists $y^*\in S$ such that $(x^*,y^*)$ is a game-stationary point.  Conversely, if $(x^*, y^*)$ is a game-stationary point of problem \eqref{Prob} ,  then $x^*$ is an optimization-stationary point.
\end{proposition}
\begin{proof}
    Suppose that $x^*$ is an optimization-stationary point of problem \eqref{Prob}. By Proposition \ref{pro:value function}, $\Phi(x)+h(x)+L_x\|x-x^*\|^2$ is strongly convex. {As $0\in N_{x^*}\mathcal{M}+\partial\Phi(x^*)+\partial h(x^*)$, and $N_{x^*}{\mathcal{M}}=N_{0}({T_{x^*}{\mathcal{M}}})$,} then $x^*$ is the unique minimum point of the problem 
    \[\min_{x\in A_{x^*}\mathcal{M}} \Phi(x)+h(x)+L_x\|x-x^*\|^2 \]
    with $A_{x^*}\mathcal{M}:=\{x\mid x=x^*+v,~v\in T_{x^*}{\mathcal{M}}\}$. 
    Therefore, by Sion minimax theorem \cite{sionGeneralMinimaxTheorems1958}, there exists $y^*\in S$ such that 
    \begin{align*}
        &f(x^*,y^*)+h(x^*)-g(y^*)\\=&\inf_{\substack{x\in A_{x^*}{\mathcal{M}},\\ \|x-x^*\|\leq 1}}\sup_{y\in S} f(x,y)+h(x)-g(y)+L_x\|x-x^*\|^2\\
        =&\sup_{y\in S}\inf_{\substack{x\in A_{x^*}{\mathcal{M}},\\ \|x-x^*\|\leq 1}} f(x,y)+h(x)-g(y)+L_x\|x-x^*\|^2,
    \end{align*} which implies  
       {  \begin{align*}
	&0\in \nabla_x f(x^*,y^*)+\partial h(x^*)+N_{x^*}\mathcal{M}, \\  &0\in \nabla_y f(x^*, y^*)-\partial g(y^*)-N_{y^*}S. 
   \end{align*}  }
Conversely, suppose that $(x^*, y^*)$ is a game-stationary point of problem \eqref{Prob}. Then, according to Proposition \ref{pro:value function}, $y^*=\mathop{\mathrm{argmax}}_{y\in S} f(x^*,y)-g(y)$ and $\nabla_x f(x^*,y^*)\in \partial \Phi(x^*)$. Thus, {$0\in\partial \Phi(x^*)+\partial h(x^*)+N_{x^*}\mathcal{M}$. }
\end{proof}   

Next, we introduce the game stationarity measure. For $(x,y)\in\mathbb{E}_1\times\mathbb{E}_2$, the game-stationarity measure is defined as {\begin{align}\label{game-stationary measure}
\mathcal{G}^{\beta}(x,y):=\max\{\|\beta u\|,~\mathrm{dist}(0,\nabla_y f(x,y)-\partial g(y)-N_{y}S )\},
\end{align}}
where $\beta>0$ and $u= \mathop{\mathrm{argmin}}_{v\in T_{x}\mathcal{M}} \langle \nabla_x f(x,y), v\rangle+h(x+v)+\frac{\beta}{2}\|v\|^2$. This definition is partly motivated by the stationarity  measure used in the Euclidean minimax problem \cite{luHybridBlockSuccessive2020}. 
 It is straightforward to verify  that $(x^*,y^*)$ is  a game-stationary point of problem \eqref{Prob} if and only if there exists $\beta>0$ such that
$\mathcal{G}^{\beta} (x^*,y^*)=0$. Using this measure, we can also introduce the notion of $\varepsilon$-game-stationary point.
\begin{definition}\label{measure}
Let $\varepsilon>0$ be given.
    We say $(x^*,y^*)$ is an $\varepsilon$-game-stationary point of problem \eqref{Prob} if there exists $\beta>0$ such that $\mathcal{G}^{\beta} (x^*,y^*) \leq \varepsilon$. 
\end{definition}

Finally, we examine the connections between $\varepsilon$-optimization-stationary points and $\varepsilon$-game-stationary points of problem \eqref{Prob}.

\begin{proposition}\label{epsilon-staionary}
Let $\varepsilon>0$ be given. Suppose that $(x^*, y^*)$ is an $\varepsilon$-game-stationary point of problem \eqref{Prob} with $\mathcal{G}^{\beta}(x^*,y^*)\leq \varepsilon$.  Then, $x^*$ is a $(C\varepsilon)$-optimization-stationary point with \[C=\max\{4\sigma_y^2+1+2L_y\sigma_y/\beta,1/\beta,1+L_x/\beta\}, \]
where $L_x,~L_y,~\sigma_y$ are defined in Assumption \ref{Assumption}.
\end{proposition}
\begin{proof}
    Suppose the point $(x^*, y^*)$ is an $\varepsilon$-game-stationary point of problem \eqref{Prob}, and let $u^*:= \arg\min_{v\in T_{x^*}\mathcal{M}} \langle \nabla_x f(x^*,y^*), v\rangle+h(x^*+v)+\frac{\beta}{2}\|v\|^2$. Then, $\|\beta u^*\|\leq\varepsilon$, and there exists $z^*\in\mathbb{E}_2$ such that $\|z^*\|\leq1$ and 
    {\begin{align}\label{optimal-y-prop4}
        \varepsilon z^* \in \nabla_y f(x^*,y^*)-\partial g(y^*)-N_{y^*}S.
    \end{align}  }
    We first prove that 
    \begin{align}\label{sub-epsilon-x+u}
        \nabla_x f(x^*+u^*,y^*)\in\partial_{\hat{\varepsilon}} \Phi(x^*+u^*) \text{ with } \hat{\varepsilon}=\varepsilon(4\sigma_y^2+1)+2L_y\sigma_y\|u^*\|.
    \end{align} 
    Consider the auxiliary function
    \begin{align*}
        \Phi_{\varepsilon}^{u^*}(x):=&\max_{y\in S} f(x+u^*,y)-g(y)-\frac{\varepsilon}{2}\|y-(y^*-z^*)\|^2\\&+\langle \nabla_y f(x^*,y^*)-\nabla_y f(x^*+u^*,y^*),y\rangle.
    \end{align*} 
    By \eqref{optimal-y-prop4}, the point $y^*$ satisfies the first-order optimality condition for the maximization problem defining $\Phi^u_{\varepsilon}(x^*)$.  Hence, applying Danskin's theorem \cite[Theorem~D1]{bernhardTheoremDanskinApplication1995}, we obtain $\nabla \Phi_{\varepsilon}^{u^*}(x^*)=\nabla_x f(x^*+u^*,y^*)$.

  Following an argument similar to the proof of Proposition \ref{pro:value function}, we deduce that $\Phi_{\varepsilon}^{u^*}$ is $L_x$-weakly convex. Consequently, for any $x \in \mathbb{E}$, 
  \begin{equation}\label{ineq-phi-u-0}
      \begin{aligned}
        &\Phi_{\varepsilon}^{u^*}(x^*)+\langle \nabla_x f(x^*+u^*,y^*),x-(x^*+u^*)\rangle\\
        \leq &\Phi_{\varepsilon}^{u^*}(x-u^*)+\frac{L_x}{2}\|x-(x^*+u^*)\|^2.
    \end{aligned} 
  \end{equation}
     Recall that $\sigma_y$ is the upper bound of $S$ and $\nabla_y f(x,y)$ is $L_y$-Lipschitz continuous on $\mathcal{M}\times S$. Using $\|z^*\| \leq 1$, we obtain the estimates
     \begin{align*}
         &\frac{\varepsilon}{2}\|y-(y^*-z^*)\|^2\leq \varepsilon(4\sigma_y^2+1),\\
         &\langle \nabla_y f(x^*,y^*)-\nabla_y f(x^*+u^*,y^*),y\rangle\leq L_y\sigma_y\|u^*\|.
     \end{align*}  
   From the definitions of $\Phi$ and $\Phi_{\varepsilon}^{u^*}$, these inequalities yield
\begin{align}\label{ineq-phi-u-1}
    \Phi(x^*+u^*)\leq \Phi_{\varepsilon}^{u^*}(x^*)+\varepsilon(4\sigma_y^2+1)+L_y\sigma_y\|u^*\|  ,
\end{align}
and for any $x\in\mathbb{E}$, 
    \begin{align}\label{ineq-phi-u-2}
        \Phi_{\varepsilon}^{u^*}(x-u^*)\leq \Phi(x) +L_y\sigma_y\|u^*\|.
    \end{align}
    Combining \eqref{ineq-phi-u-0}, \eqref{ineq-phi-u-1} and \eqref{ineq-phi-u-2}, we obtain that for any $x\in\mathbb{E}$,
    \begin{align*}
        &\Phi(x^*+u^*)+\langle \nabla_x f(x^*+u^*,y^*),x-(x^*+u^*)\rangle-\varepsilon(4\sigma_y^2+1)-2L_y\sigma_y\|u^*\|\\ \leq &\Phi(x)+\frac{L_x}{2}\|x-(x^*+u^*)\|^2.
    \end{align*}
   By the definition of the $\varepsilon$-subdifferential $\partial_{\varepsilon}\Phi(\cdot)$, this proves \eqref{sub-epsilon-x+u}. 
   
   On the other hand, from the definition of $u^*$, we have
   {\[-\beta u^*-\nabla_x f(x^*,y^*)\in \partial h(x^*+u^*)+N_{x^*}\mathcal{M}.\]}
   Combining this with \eqref{sub-epsilon-x+u} and recalling that $\nabla_x f(\cdot,y)$ is $L_x$-Lipschitz continuous, we deduce that 
   { \begin{align*}
        \mathrm{dist}(0,\partial_{\hat{\varepsilon}}\Phi(x^*+u^*)+\partial h(x^*+u^*)+N_{x^*}\mathcal{M})<(\beta+L_x)\|u^*\|.
    \end{align*}}
    Since $\|\beta u^*\| \leq \varepsilon$ by the $\varepsilon$-game-stationarity condition, it follows that $\hat{\varepsilon}\leq \varepsilon(4\sigma_y^2+1+2L_y\sigma_y/\beta)$ and $x^*$ is a $(C\varepsilon)$-optimization-stationary point with $C=\max\{4\sigma_y^2+1+2L_y\sigma_y/\beta,1/\beta,1+L_x/\beta\}$. 
\end{proof}

\section{The Proposed MPGDA-PA Algorithm}\label{sec:4}
In this section, we propose the MPGDA-PA method for solving problem \eqref{Prob} and establish its iteration complexity. Our algorithm development is partially motivated by \cite{heApproximationProximalGradient2024,luHybridBlockSuccessive2020,xuRiemannianAlternatingDescent2024}. For the $k$-th iteration, we introduce the following value function:
\begin{align}\label{def-Qk}
Q_k(x) := h(x) +\Phi_k(x)
\end{align}
where $\Phi_k: \mathbb{E}_1 \to \mathbb{R}$ is defined  by
\begin{align}\label{def-Phi}
\Phi_k(x) := \max_{y\in S} \{f(x,y)-g(y)-\frac{\gamma_k}{2}\|y\|^2-\frac{\rho_k}{2}\|y-y_k\|^2\}.
\end{align}
Here, $y_k$ denotes the iterate with respect to $y$ obtained in the previous iteration,
$ \gamma_k> 0$ is the regularization parameter, and $\rho_k > 0$ is the proximal parameter. Intuitively, the  regularization term $\frac{\gamma_k}{2}\|y\|^2$ ensures strong concavity of the optimization problem in \eqref{value function}, while the proximal term $\frac{\rho_k}{2}\|y-y_k\|^2$ stabilizes the update of $y$. Moreover, we choose $\gamma_k=\gamma_0/k^{1/3}$ with some $\gamma_0>0$ and assume the non-increasing sequence $\{\rho_k\}$ to be summable, i.e., \[\{\rho_k\}\in\mathcal{S}:=\{\{\rho_k\}\mid\sum_{k=0}^\infty \rho_k<+\infty,~~\rho_{k}\geq \rho_{k+1}>0, k=0,1,\dots~\}.\]
We further introduce the mapping $\overline{y}_k:\mathbb{E}_1\to \mathbb{E}_2$ defined by 
\begin{align}\label{y-bar}
    \overline{y}_k(x):=\mathop{\mathrm{argmax}}_{y\in S} \{f(x,y)-g(y)-\frac{\gamma_k}{2}\|y\|^2-\frac{\rho_k}{2}\|y-y_k\|^2\}.
\end{align} 
Then it holds that  $\Phi_k(x)=f(x,\overline{y}_k(x))-\frac{\gamma_k}{2}\|\overline{y}_k(x)\|^2-\frac{\rho_k}{2}\|\overline{y}_k(x)-y_k\|^2$ and in view of Danskin's theorem \cite[Theorem~D1]{danskinTheoryMaxMinApplications1966}, $\Phi_k$ is smooth on $\mathbb{E}_1$ and $\nabla \Phi_k(x)=\nabla_x f(x,\overline{y}_k(x))$. 

Now, at the $k$-th iteration, the proposed MPGDA-PA method computes the new iterate by approximately solving the subproblem:
\begin{equation}
\min_{x \in \mathcal{M}} Q_k(x).
\end{equation}
Inspired by the structure of $Q_k$ and the ManPG algorithm \cite[Algorithm~4.1]{chenProximalGradientMethod2020}, we perform Riemannian manifold proximal gradient descent steps for $T_k$ times to generate $x_{k+1}$ with $T_k$ being an integer in $[1, \overline{T}]$ where $\overline{T}$ is a predetermined positive integer. Then $y_{k+1}$ is updated as $\overline{y}_k(x_{k+1})$ with $\overline{y}_k$ defined in \eqref{y-bar}. Note that the update of $y$ can also be interpreted as performing a proximal ascent step for the  function $f(x_{k+1},y)-g(y)-\frac{\gamma_k}{2}\|y\|^2$ on $S$. We formally present the proposed method in Algorithm \ref{MPGDA-PA}.

\begin{algorithm}[!ht]
 	\caption{Manifold Proximal Gradient Descent Ascent (MPGDA-PA) for problem \eqref{Prob}}
     \label{MPGDA-PA}
    \begin{algorithmic}[1] 
        \REQUIRE   $c_1\in(0,1)$, $\eta\in(0,1)$, $\gamma_{0}>0$, $x_0\in \mathcal{M}$, $y_0\in S $, $0< l_{\min}<l_{\max}$, $\{\rho_k\}\in \mathcal{S}$, $\{\gamma_k=\gamma_0/k^{1/3}: k\in\mathbb{N}\}$, $\{T_k\}\subset [1,\overline{T}]$. 
         \FOR{$k=0,1, \dots$}
         \STATE Set $x_{k,0}=x_k$.
         \FOR{$i=0,\dots, T_k-1$}
          \STATE Choose $l_{k,i}\in [l_{\min},l_{\max}]$ and set $\beta_{k,i}=l_{k,i}/(\rho_{k}+\gamma_{k})$.
         \STATE  Compute \begin{align}\label{v ki}
             v_{k,i}:=\mathop{\mathrm{argmin}}_{v\in T_{x_{k,i}}\mathcal{M}} \langle \nabla_x f(x_{k,i}, \overline{y}_k(x_{k,i})), v\rangle+h(x_{k,i}+v)+\frac{\beta_{k,i}}{2}\left\|v\right\|^2,
         \end{align}
         \STATE Find the smallest non-negative integer $j$ such that 
         \begin{align}\label{eq:sufficient descent}
             Q_{k}(R_{x_{k,i}}(\eta^j v_{k,i}))\leq Q_{k}(x_{k,i})-c_1\eta^{j}\beta_{k,i}\|v_{k,i}\|^2+2\rho_k\sigma_y^2
         \end{align}
         and update $x_{k,i+1}=R_{x_{k,i}}(\eta^{j}v_{k,i})$.    
        \ENDFOR   
        \STATE $x_{k+1}=x_{k,T_k},~y_{k+1}=\overline{y}_k(x_{k+1})$.
         \ENDFOR     
    \STATE \textbf{return} $(x_{k+1},y_{k+1})$.
    \end{algorithmic}
\end{algorithm}

Before proceeding, we make some remarks on the iteration procedure of the MPGDA-PA algorithm. First, given $x_{k,i} \in \mathcal{M}$, the MPGDA-PA algorithm computes a descent direction $v_{k,i}$ of $Q_k$ restricted to the tangent space $T_{x_{k,i}}\mathcal{M}$ via tackling the linear constrained strongly convex optimization problem \eqref{v ki}, which can usually be efficiently solved, see \cite[Section~4]{li2024proximal}, for example. Second, for all $x\in\mathcal{M}$,  $\overline{y}_k(x)$ defined in \eqref{y-bar} can be solved by applying FISTA \cite{Beck2017First} to problem \eqref{def-Phi}, which converges linearly due to the strong convexity and smoothness of 
\[\frac{\gamma_k}{2}\|\cdot\|^2+\frac{\rho_k}{2}\|\cdot-y_k\|^2-f(x,\cdot).\]
In particular, if $f(x,y)$ is linear with respect to $y$ for fixed $x$, i.e., $f(x,y) = \langle \mathcal{A}(x), y \rangle$, 
then $\overline{y}_k(x)$ can be further simplified into a proximal problem like \eqref{proximal-g}  as follows:\[
\overline{y}_k(x) = \mathop{\mathrm{argmax}}_{y \in S} \left\{  -g(y) - \frac{\gamma_k + \rho_k}{2} \|y-\frac{\rho_k y_k+\mathcal{A}(x)}{\rho_k+\gamma_k}\|^2 \right\}.
\]
 Third, in practice, we make some choices of parameters $l_{k,i}$ and $\rho_k$, which enhance the efficiency of the MPGDA-PA algorithm in our test and also satisfy $l_{k,i} \in [l_{\min}, l_{\max}]$ and $\rho_k \in S$. Motivated by the Riemannian Barzilai-Borwein (BB) stepsize \cite{iannazzo2018riemannian,2013AWen}, we choose  $l_{k,i}$ as follows:
\[l_{k,i}=\begin{cases}
    \min\{\max\{l_{\min},(\rho_k+\gamma_k)|\frac{\langle \Delta X, \Delta R\rangle}{\|\Delta X\|^2}|)\},l_{\max}\},& if~\Delta X\neq 0,\\
    l_{\max},& if ~\Delta X=0,
\end{cases}\]
where $\Delta X=X_{k,i}-X_{k,i-1}$ and {$\Delta R=\mathrm{grad} \Phi_k(x_{k,i})-\mathrm{grad}\Phi_k(x_{k,i-1})$}.
Moreover, inspired by \cite{xuRiemannianAlternatingDescent2024}, for given constant $\theta>1$, $\tau_1\in (0,1)$ and $\tau_2\in (0,1)$, we set $\rho_0=\xi_0$ and for $k\geq 1$,
\begin{align}
    \rho_k=\frac{\xi_k}{k^\theta}~~\text{with}~~\xi_k=\begin{cases}
        \tau_2 \xi_{k-1},&\text{if}~~ \delta_k\geq \tau_1\delta_{k-1},\\
        \xi_{k-1},& \text{else},
    \end{cases}
\end{align}
where $\delta_k$ is defined as \[\delta_k=\|\gamma_{k-1}y_{k}+\rho_{k-1}(y_{k}-y_{k-1})\|_{\infty}.\]

In what follows, we conduct convergence analysis for the MPGDA-PA algorithm. First, we present a technical lemma that guarantees the boundedness of the mapping $\overline{y}_k$.
\begin{lemma}\label{concave1}
Let the mapping $\overline{y}_k:\mathbb{E}_1\to\mathbb{E}_2$ be defined in \eqref{y-bar} for $k\in\mathbb{N}$. Then for all $x$, $\bar{x}\in\mathbb{E}_1$ and $y\in\mathbb{E}_2$ , it holds that   
    \begin{align*}
       &f(x, y)-g(y)\\&\leq  f(x, \overline{y}_k(\bar{x}))-g(\overline{y}_k(\bar{x}))-\frac{\rho_k}{2}\|\overline{y}_k(\bar{x})-y_k\|^2-\frac{\gamma_k}{2}\|\overline{y}_k(\bar{x})\|^2+\frac{\gamma_k}{2}\|y\|^2\\
       &+\frac{\rho_k}{2}\|y-y_k\|^2+\frac{L_y^2}{2(\rho_k+\gamma_k)}\|x-\bar{x}\|^2.     
   \end{align*}  
\end{lemma}
\begin{proof}
     By the definition of $\overline{y}_k(\bar{x})$, we have {\[0\in \nabla_y f(\bar{x}, \overline{y}_k(\bar{x}))- \partial g(\overline{y}_k(\bar{x}))-N_{\overline{y}_k(\bar{x})}S-\rho_k(\overline{y}_k(\bar{x})-y_k)-\gamma_k \overline{y}_k(\bar{x}). \]}
     This and the concavity of $f(x,\cdot)-g(\cdot)$ yield that  \begin{align*}
            &f(x, y)-g(y)\\
        &\leq f(x, \overline{y}_k(\bar{x}))-g(\overline{y}_k(\bar{x}))+ \langle\nabla_y f(x, \overline{y}_k(\bar{x}))-\nabla_y f(\bar{x},\overline{y}_k(\bar{x})), y-\overline{y}_k(\bar{x})\rangle\\
        &+\langle \rho_k(\overline{y}_k(\bar{x})-y_k)+\gamma_k \overline{y}_k(\bar{x}), y-\overline{y}_k(\bar{x})\rangle. 
        \end{align*}   
     Since $2\langle a, a-b\rangle=\left\| a\right\|^2+\left\|a-b\right\|^2-\left\|b\right\|^2$ and $\nabla_y f(x,y)$ is $L_y$-Lipschitz continuous,  it follows that, 
     \begin{align*}
       &f(x, y)-g(y)\\&\leq  f(x, \overline{y}_k(\bar{x}))-g(\overline{y}_k(\bar{x}))-\frac{\rho_k}{2}\|\overline{y}_k(\bar{x})-y_k\|^2-\frac{\gamma_k}{2}\|\overline{y}_k(\bar{x})\|^2+\frac{\gamma_k}{2}\|y\|^2\\
       &+\frac{\rho_k}{2}\|y-y_k\|^2+\frac{L_y^2}{2(\rho_k+\gamma_k)}\|x-\bar{x}\|^2.     
   \end{align*}
   
\end{proof}

Next, we prove that the descent condition \eqref{eq:sufficient descent} must be satisfied after a unified finite number of iterations for all $k\in\mathbb{N}$ and $i\in[0,T_k-1]$, which ensures that the proposed Algorithm \ref{MPGDA-PA} is well-defined.
\begin{proposition}\label{prop:sufficient descent}
     Consider Algorithm \ref{MPGDA-PA}. There exists a constant $\bar{\alpha}>0$ such that for all  $k\in\mathbb{N}$ and $0<\alpha\leq \min\{1,~\bar{\alpha}\}$, it holds that \[Q_{k}(R_{x_{k,i}}( \alpha v_{k,i}))\leq Q_{k}(x_{k,i})-\frac{\alpha}{4}\beta_{k,i}\|v_{k,i}\|^2+2\rho_k\sigma_y^2\]
     for all $i=0,1,\dots,T_k-1$.        
\end{proposition}
   
\begin{proof}
Let $G:=\sup\{ \|\nabla_x f(x,y)\|\mid x\in\mathcal{M},y\in S\}<\infty.$ Assume that $0<\alpha\leq\min\{1,~\frac{\beta_{k,i}}{2M_2(L_h+G)+M_1^2L_x}\}$ with $M_1>0,~M_2>0$ being defined in Proposition \ref{retr} and $x_{\alpha}=R_{x_{k,i}}(\alpha v_{k,i})$. Then, following an argument similar to the proof of \cite[Lemma~5.2]{chenProximalGradientMethod2020}, we deduce that 
\[f(x_{\alpha}, \overline{y}_k(x_{k,i}))+h(x_{\alpha})\leq f(x_{k,i}, \overline{y}_k(x_{k,i}))+h(x_{k,i})-\frac{\alpha \beta_{k,i}}{2}\|v_{k,i}\|^2.\]  

 According to Lemma \ref{concave1}, we have 
\begin{align*}
       &f(x_{\alpha}, \overline{y}_k(x_{\alpha}))-g(\overline{y}_k(x_{\alpha}))\\&\leq  f(x_{\alpha}, \overline{y}_k(x_{k,i}))-g(\overline{y}_k(x_{k,i}))-\frac{\rho_k}{2}\|\overline{y}_k(x_{k,i})-y_k\|^2-\frac{\gamma_{k}}{2}\|\overline{y}_k(x_{k,i})\|^2\\
       &+\frac{\gamma_{k}}{2}\|\overline{y}_k(x_{\alpha})\|^2+\frac{\rho_k}{2}\|\overline{y}_k(x_{\alpha})-y_k\|^2+\frac{L_y^2}{2(\rho_k+\gamma_{k})}\|x_{\alpha}-x_{k,i}\|^2.     
   \end{align*}
   By summing the above two inequalities and the inequality \[\|x_{\alpha}-x_{k,i}\| \leq \alpha M_1\|v_{k,i}\|\] from Proposition \ref{retr},
     we derive the following inequality:
     \begin{equation}
         \begin{aligned} \label{ineq-F-gammak}      &F(x_{\alpha},\overline{y}_k(x_{\alpha}))-\frac{\gamma_{k}}{2}\|\overline{y}_k(x_{\alpha})\|^2\\&\leq  F(x_{k,i},\overline{y}_k(x_{k,i}))-\frac{\rho_k}{2}\|\overline{y}_k(x_{k,i})-y_k\|^2-\frac{\gamma_{k}}{2}\|\overline{y}_k(x_{k,i})\|^2\\&+\frac{\rho_k}{2}\|\overline{y}_k(x_{\alpha})-y_k\|^2-(\frac{\alpha\beta_{k,i}}{2}-\frac{\alpha^2M_1^2L_y^2}{2(\rho_k+\gamma_{k})})\|v_{k,i}\|^2.  
   \end{aligned}
     \end{equation}
 Since $\beta_{k,i}=l_{k,i}/(\rho_k+\gamma_{k})$, for any $0<\alpha\leq \frac{l_{k,i}}{2M_1^2L_y^2}$,  it holds that 
   \begin{align}\label{ineq-beta-v}
       (\frac{\alpha\beta_{k,i}}{2}-\frac{\alpha^2M_1^2L_y^2}{2(\rho_k+\gamma_k)})\|v_{k,i}\|^2&=\beta_{k,i}(\frac{\alpha}{2}-\frac{\alpha^2M_1^2L_y^2}{2l_{k,i}})\| v_{k,i}\|^2\geq \frac{\alpha}{4}\beta_{k,i}\| v_{k,i}\|^2.    
   \end{align}
Combining \eqref{ineq-F-gammak}, \eqref{ineq-beta-v}, and the definition of $Q_k(x)$, we obtain that for all $k\geq 0$ and any $0<\alpha\leq \min\{1,\bar{\alpha}\}$,
\[Q_{k}(x_{\alpha})\leq Q_{k}(x_{k,i})-\frac{\alpha}{4}\beta_{k,i}\|v_{k,i}\|^2+2\rho_k\sigma_y^2,\]
 where  $\bar{\alpha}=\min\{\frac{l_{\min}}{(\rho_0+\gamma_0)(2M_2(L_h+G)+M_1^2L_x)},\frac{l_{\min}}{2M_1^2L_y^2}\}$.  
\end{proof}
By Proposition \ref{prop:sufficient descent}, the backtracking line search procedure in Algorithm \ref{MPGDA-PA} will terminate after a finite number of steps. We present this result in the following proposition.
\begin{proposition}\label{prop:line-search}
    For all $k\geq 0$, the condition \eqref{eq:sufficient descent} in Algorithm \ref{MPGDA-PA} is satisfied within at most $J_1$ backtracking steps, where
    \[ J_1:=\max\left\{\left\lceil \log_{\eta} \frac{\min\{1,\bar{\alpha}\}}{4c_1}\right\rceil,0\right\} \text{ with } \bar{\alpha}\text{ being defined in Proposition \ref{prop:sufficient descent}.}\]
\end{proposition}

We are now ready to establish the iteration complexity for Algorithm \ref{MPGDA-PA}. Let $\varepsilon>0$ be a given target accuracy, we define \begin{align}\label{complex-measure}
    T(\varepsilon):=\min\{k\mid \mathcal{G}^{\beta_{k}}(x_{k},y_k)<\varepsilon\}
\end{align}
where $\mathcal{G}^{\beta}(\cdot,\cdot)$ is defined in \eqref{game-stationary measure}. In other words, $T(\varepsilon)$ stands for the minimal number of outer iterations needed for Algorithm \ref{MPGDA-PA} to obtain a $\varepsilon$-game-stationary point.  We further introduce two constants $\overline{Q}$ and $\underline{Q}$ as follows
\begin{align*}
    \overline{Q}&:=\sup\{{F(x,y)\mid x\in\mathcal{M},y\in S}\},\\ \underline{Q}&:=\inf\{F(x,y)-\frac{\rho_0}{2}\|y-y_0\|^2-\frac{\gamma_0}{2}\|y\|^2\mid x\in\mathcal{M},y\in S\}.
\end{align*}

\begin{theorem}\label{the:converge}
     Given $\varepsilon>0$, let $\{(x_{k},y_{k})\}$ be the sequence generated by Algorithm \ref{MPGDA-PA}.  Then, the following inequality holds:
     \[T(\varepsilon)\leq \frac{(\max\{(2\rho_0+\gamma_0)\sigma_y ,C\})^3}{\varepsilon^{3}}+1\]
     with
     \[C=\frac{\sqrt{2(\overline{Q}-\underline{Q}+\frac{\gamma_0}{2}\sigma_y^2+2(\overline{T}+1)S_{\rho}\sigma_y^2)l_{\max}}}{\sqrt{3c_1\eta^{J_1}\gamma_0}}.\]
Here, the constants $c_1,~\eta,~\rho_0,~\gamma_0,~\overline{T}$, and $l_{\max}$ are parameters of Algorithm \ref{MPGDA-PA}, $J_1$ is defined in Proposition \ref{prop:line-search}, and $S_{\rho}:=\sum_{k=0}^{\infty}\rho_k<\infty$.
     
  \end{theorem}
    \begin{proof}
    Let $\beta_k:=\beta_{k,0},~v_k:=v_{k,0}$. Since $\gamma_{k+1}\leq \gamma_k$, $\rho_{k+1}\leq \rho_k$, and $x_{k+1}=x_{k,T_k}$, by the definition of $Q_k$ and $\sigma_y$, we have
    \begin{align*}
        &Q_{k+1}(x_{k+1})\\=&h(x_{k+1})+\max_{y\in S} f(x_{k+1},y)-g(y)-\frac{\gamma_{k+1}}{2}\|y\|^2-\frac{\rho_{k+1}}{2}\|y-y_{k+1}\|^2 \\
        \leq& h(x_{k,T_k})+\max_{y\in S} f(x_{k, T_k},y)-g(y)-\frac{\gamma_k}{2}\|y\|^2-\frac{\rho_k}{2}\|y-y_k\|^2\\&+\frac{\gamma_k-\gamma_{k+1}}{2}\sigma_y^2+2\rho_k\sigma_y^2\\
     =&Q_k(x_{k,T_k})+\frac{\gamma_k-\gamma_{k+1}}{2}\sigma_y^2+2\rho_k\sigma_y^2.      
    \end{align*}
Combining this with \eqref{eq:sufficient descent} and Proposition \ref{prop:line-search}, we obtain that 
\begin{align*}
    Q_{k+1}(x_{k+1})\leq Q_k(x_{k})+2(T_k+1)\rho_k\sigma_y^2+\frac{\gamma_k-\gamma_{k+1}}{2}\sigma_y^2-c_1\eta^{J_1}\beta_k\|v_k\|^2.
\end{align*}
Summing the above inequality over $k=K_1, K_1+1, \dots, K$ yields 
\begin{align}\label{the:ineq1}
    \sum_{k=K_1}^{K}c_1\eta^{J_1}\beta_k\|v_{k}\|^2\leq {Q}_{K_1}(x_{K_1})-{Q}_{K}(x_{K+1})+\frac{\gamma_{K_1}-\gamma_{K}}{2}\sigma_y^2+2(\overline{T}+1)S_{\rho}\sigma_y^2
\end{align} 
 with $S_\rho=\sum_{k=0}^{\infty}\rho_k<\infty$. Since $\beta_k\leq \frac{l_{\max}}{\gamma_0k^{-1/3}}$, then
\begin{align}\label{the:ineq2}
    \sum_{k=K_1}^{K}c_1\eta^{J_1}\beta_k\|v_{k}\|^2\geq (\frac{c_1\eta^{J_1}\gamma_0 }{l_{\max}}\sum_{k=K_1}^{K}\frac{1}{k^{1/3}})\min_{K_1\leq k\leq K}\|\beta_k v_k\|^2.
\end{align}
Note that 
\begin{align*}
    \sum_{k=K_1}^{K} \frac{1}{k^{1/3}} \geq \int_{K_1}^{K+1} \frac{1}{x^{1/3}} dx=\frac{3}{2}((K+1)^{2/3}-K_1^{2/3}),
\end{align*} let $K=\lceil{2^{3/2}K_1}\rceil$ ,  then we obtain
\begin{align}\label{the:ineq3}
  \sum_{k=K_1}^{\lceil{2^{3/2}K_1}\rceil}\frac{1}{k^{1/3}} \geq \frac{3}{2}K_1^{2/3}.
\end{align}
Therefore, combining \eqref{the:ineq1}, \eqref{the:ineq2}, \eqref{the:ineq3} and the bounds $Q_{K}(x_{K_1})\leq \overline{Q}$ and $Q_{\lceil{2^{3/2}K_1}\rceil+1}(x_{\lceil{2^{3/2}K_1}\rceil+1})\geq \underline{Q}$,  we have \begin{align}\label{the:ineq4}
    \min_{K_1\leq k\leq \lceil{2^{3/2}K_1}\rceil} \|\beta_k v_k\|^2\leq \frac{2l_{\max}}{3c_1\eta^{J_1}\gamma_0 K_1^{2/3}}(\overline{Q}-\underline{Q}+\frac{\gamma_0}{2}\sigma_y^2+2(\overline{T}+1)S_{\rho}\sigma_y^2).
\end{align}
In view of definition \eqref{y-bar} and the associated first-order optimality condition, for $\overline{y}_k(x_{k,0})$, we have that
{\begin{align*}
    &\rho_k(\overline{y}_k(x_{k,0})-y_k)+\gamma_k \overline{y}_k(x_{k,0})\\\in&\nabla_y f(x_{k,0},\overline{y}_k(x_{k,0}))-\partial g(\overline{y}_k(x_{k,0}))-N_{\overline{y}_k(x_{k,0})}S.
\end{align*} }
 Invoking this inclusion, we further obtain that
{\begin{equation}\label{the:ineq5}
    \begin{aligned}
    &\mathrm{dist}(0,\nabla_y f(x_{k,0},\overline{y}_k(x_{k,0}))-\partial g(\overline{y}_k(x_{k,0}))-N_{\overline{y}_k(x_{k,0})}S)\\
    \leq &\|\rho_k(\overline{y}_k(x_{k,0})-y_k)+\gamma_k \overline{y}_k(x_{k,0})\|\\
    \leq &\sigma_y(2\rho_k+\gamma_k).
\end{aligned}
\end{equation}}
Finally, by the definition of $\rho_k$, $\gamma_k$ and $\mathcal{G}^{\beta_k}$ , combining \eqref{the:ineq4} and \eqref{the:ineq5}, we  conclude that \[\min_{K_1\leq k\leq \lceil 2^{3/2}K_1\rceil }\mathcal{G}^{\beta_k}(x_{k,0},\overline{y}_k(x_{k,0}))\leq \frac{\max\{(2\rho_0+\gamma_0)\sigma_y ,C\}}{K_1^{1/3}}\]
with
\[C=\frac{\sqrt{2l_{\max}(\overline{Q}-\underline{Q}+\frac{\gamma_0}{2}\sigma_y^2+2(\overline{T}+1)S_{\rho}\sigma_y^2)}}{\sqrt{3c_1\eta^{J_1}\gamma_0}},\]  
and the proof is finished. 
    \end{proof}

Theorem \ref{the:converge} demonstrates that the proposed MPGDA-PA algorithm can find an $\varepsilon$-game-stationary point for problem \eqref{Prob} within $\mathcal{O}(\varepsilon^{-3})$ outer iterations. In view of Proposition \ref{epsilon-staionary}, this also indicates  that the outer iteration complexity of the MPGDA-PA algorithm for returning an $\varepsilon$-optimization-stationary point is $\mathcal{O}(\varepsilon^{-3})$.

\section{The Proposed MPGDA-PGA Algorithm}\label{sec:5}
In this section, we propose the MPGDA-PGA algorithm and establish its iteration complexity. In contrast to the MPGDA-PA algorithm, the MPGDA-PGA algorithm performs a single proximal gradient ascent step to update $y$, which significantly reduces the computational cost of each iteration. 

Motivated in part by \cite{xuUnifiedSingleloopAlternating2023}, we develop the MPGDA-PGA algorithm by alternatively performing one manifold proximal gradient descent step and a proximal gradient ascent step for a regularized version of the original function, i.e.,
\begin{align*}
  F_k(x,y):=  f(x,y)+h(x)-g(y)-\frac{\gamma_{k-1}}{2}\|y\|^2
\end{align*}
where $\gamma_{k-1}> 0$ is the regularization parameter. More specifically, for a given pair $(x_k,y_k)\in \mathcal{M}\times S$ at the $k$-th iteration, the proposed MPGDA-PGA first computes a manifold proximal gradient direction $v_k$ similarly as in MPGDA-PA, that is:
\begin{align}\label{update-v-GA}
    v_k:= \mathop{\mathrm{argmin}}_{v\in T_{x_k}\mathcal{M}}\langle \nabla_x f(x_k, y_k), v\rangle + h(x_k+v) + \frac{\beta_k}{2}\|v\|^2,
\end{align} 
where $\beta_k>0$ is the proximal parameter. Then, it generates $x_{k+1}:=R_{x_k}(\alpha_k v_k)$ for some appropriate $\alpha_k\in (0,1]$, and accordingly set $y_{k+1}:=\hat{y}_k(x_{k+1})$,  where the mapping $\hat{y}_k:\mathbb{E}_1\to\mathbb{E}_2$ is defined by 
\begin{align}\label{y-update-GA}
\hat{y}_{k}(x) := \mathop{\mathrm{argmax}}_{y\in S} \left\langle \nabla_y f(x, y_k)-\gamma_k y_k, y-y_k\right\rangle - \frac{1}{2\rho}\|y-y_k\|^2 - g(y)
\end{align} 
with $\rho>0$ being the proximal parameter. Before we propose the  MPGDA-PGA algorithm formally, we establish the following lemma regarding the difference of $F_{k+1}(x_{k+1},y_{k+1})$ and $F_k(x_k,y_k)$.
\begin{lemma}\label{lem-desc-F-GA}
  Let $x_k\in\mathcal{M},~y_{k-1} \in S$ be given. Define $v_k$ as in \eqref{update-v-GA}, and set $x_{k+1}:=R_{x_k}(\alpha_k v_k)$ for some $\alpha_k\in(0,1]$. Further, let $y_k:=\hat{y}_{k-1}(x_k)$ and $y_{k+1}:=\hat{y}_k(x_{k+1})$ with $\hat{y}_k$ defined in \eqref{y-update-GA}. If $\rho$ and $\gamma_{k}$ satisfy 
  \begin{align}\label{range-rho,gamma}
    0<\rho\leq \frac{1-2(\kappa+1)^{-1/4}}{L_y},~~\gamma_k=\frac{2}{\rho(k+\kappa+2)^{1/4}} \text{ with } \kappa>15,
\end{align} then it holds that 
    \begin{equation}\label{ineq:F-1}
   \begin{aligned}
    &F_{k+1}(x_{k+1},y_{k+1})-F_k(x_{k},y_k)\\
    \leq &-(\beta_k-b_1\alpha_k)\alpha_k\| v_k\|^2+\frac{1}{2\rho}\|y_{k}-y_{k-1}\|^2+\frac{1}{\rho}\|y_{k+1}-y_k\|^2+\frac{\gamma_{k-1}-\gamma_k}{2}\sigma_y^2,
\end{aligned} 
\end{equation}
where \[b_1:=M_2(L_h+G)+M_1^2(L_x+\rho L_y^2)/2\]
with $G=\sup\{ \|\nabla_x f(x,y)\|\mid x\in\mathcal{M},y\in S\}<\infty$, $M_1$ and $M_2>0$ being defined in Proposition \ref{retr}.
\end{lemma}
\begin{proof}
Let $f_k(x,y):=f(x,y)-\frac{\gamma_{k-1}}{2}\|y\|^2$. From the definition of $y_{k}$, it holds that 
{\begin{align}\label{optimal-def-yk}
    0\in \nabla_y f_{k}(x_{k}, y_{k-1})- \partial g(y_{k})-N_{y_{k}}S-\frac{1}{\rho}(y_{k}-y_{k-1}).
\end{align} }
By the concavity of $f_{k+1}(x,\cdot)-g(\cdot)$, it follows that
\begin{equation}\label{ineq:y-1}
   \begin{aligned}
    &f_{k+1}(x_{k+1},y_{k+1})-g(y_{k+1})-(f_{k+1}(x_{k+1},y_k)-g(y_k))\\
    \leq& \langle \nabla_y f_{k+1}(x_{k+1},y_k)-\nabla_y f_{k}(x_k,y_{k-1})+\frac{1}{\rho}( y_k-y_{k-1}),y_{k+1}-y_k\rangle.
\end{aligned} 
\end{equation}
Note that
\begin{equation}\label{eq:nabla-inner}
   \begin{aligned}
   &\langle \nabla_y f_{k+1}(x_{k+1},y_k)-\nabla_y f_{k}(x_k,y_{k-1}),y_{k+1}-y_k\rangle.\\
    =&\langle\nabla_y f(x_{k+1},y_k)-\nabla_y f(x_k,y_{k}),y_{k+1}-y_k\rangle\\ 
    &+\langle\nabla_y f_{k}(x_{k},y_k)-\nabla_y f_{k}(x_k,y_{k-1}),y_{k}-y_{k-1}\rangle\\
    &+\langle\nabla_y f_{k}(x_{k},y_k)-\nabla_y f_{k}(x_k,y_{k-1}),w_k\rangle\\
    &-(\gamma_k-\gamma_{k-1})\langle y_k,y_{k+1}-y_k\rangle.
\end{aligned} 
\end{equation}
where $w_k:=y_{k+1}-2y_k+y_{k-1}$.
 Next, we provide bounds on the inner product terms of \eqref{eq:nabla-inner}. Firstly, invoking \cite[Theorem~5.8]{Beck2017First}, by the concavity of $f_{k}(x,\cdot)$ and the Lipschitz continuity of $\nabla_y f_{k}(x,\cdot)$,  it holds that 
\begin{equation}\label{ineq:inner-1-1}
    \begin{aligned}
   &\langle\nabla_y f_{k}(x_{k},y_k)-\nabla_y f_{k}(x_k,y_{k-1}),y_{k}-y_{k-1}\rangle\\\leq&-\frac{1}{L_y+\gamma_{k-1}}\|\nabla_y f_{k}(x_{k},y_k)-\nabla_y f_{k}(x_k,y_{k-1})\|^2\\
   \leq& -\frac{1}{L_y+\gamma_{-1}}\|\nabla_y f_{k}(x_{k},y_k)-\nabla_y f_{k}(x_k,y_{k-1})\|^2,
\end{aligned}
\end{equation}
the last inequality comes from the fact that $\gamma_{k-1}\leq\gamma_{-1}$. 
In addition, by the strong concavity of $f_{k}(x,\cdot)$, we have
\begin{equation}\label{ineq:inner-1-2}
    \begin{aligned}
   \langle\nabla_y f_{k}(x_{k},y_k)-\nabla_y f_{k}(x_k,y_{k-1}),y_{k}-y_{k-1}\rangle\leq-\gamma_{k-1}\|y_k-y_{k-1}\|^2.
\end{aligned}
\end{equation}
Combining \eqref{ineq:inner-1-1} and \eqref{ineq:inner-1-2}, it follows that
\begin{equation}\label{ineq:inner-1}
    \begin{aligned}
   &\langle\nabla_y f_{k}(x_{k},y_k)-\nabla_y f_{k}(x_k,y_{k-1}),y_{k}-y_{k-1}\rangle\\
   \leq&-\frac{\gamma_{k-1}}{2}\|y_k-y_{k-1}\|^2-\frac{1}{2(L_y+\gamma_{-1})}\|\nabla_y f_{k}(x_{k},y_k)-\nabla_y f_{k}(x_k,y_{k-1})\|^2.
\end{aligned}
\end{equation}
Secondly, by Cauchy-Schwarz inequality and the Lipschitz continuity of $\nabla_y f(x,y)$, it holds that
\begin{equation}
    \begin{aligned}
    &\langle\nabla_y f(x_{k+1},y_k)-\nabla_y f(x_k,y_{k}),y_{k+1}-y_k\rangle\\\leq&\frac{\rho L_y^2}{2}\|x_{k+1}-x_k\|^2+\frac{1}{2\rho}\|y_{k+1}-y_k\|^2,\label{ineq:inner-2}
\end{aligned}
\end{equation}
\begin{equation}
    \begin{aligned}
    &\langle\nabla_y f_{k}(x_{k},y_k)-\nabla_y f_{k}(x_k,y_{k-1}),w_k\rangle\\
    \leq &\frac{\rho}{2}\|\nabla_y f_{k}(x_{k},y_k)-\nabla_y f_{k}(x_k,y_{k-1})\|^2+\frac{1}{2\rho}\|w_k\|^2.\label{ineq:inner-3}
\end{aligned}
\end{equation}
By the range of $\rho$ and direct calculation,  it holds that $\rho<\frac{1}{L_y+\gamma_{-1}}$.
Combining this, the definition of $f_{k}(x_{k+1},y_{k+1})$,
 \eqref{ineq:y-1}, \eqref{eq:nabla-inner},  \eqref{ineq:inner-1}, \eqref{ineq:inner-2}, \eqref{ineq:inner-3}, and the following identities 
\begin{align}
    \frac{1}{\rho}\langle y_k-y_{k-1},y_{k+1}-y_k\rangle&=\frac{1}{2\rho}(\|y_{k+1}-y_k\|^2+\|y_k-y_{k-1}\|^2-\|w_k\|^2),\\  
    (\gamma_k-\gamma_{k-1})\langle y_k,y_{k+1}-y_k\rangle&=\frac{\gamma_k-\gamma_{k-1}}{2}(\|y_{k+1}\|^2-\|y_k\|^2-\|y_{k+1}-y_k\|^2),\label{eq-gamma-y}
\end{align}
we deduce that
\begin{equation}\label{ineq:y-2}
   \begin{aligned}
    &f(x_{k+1},y_{k+1})-g(y_{k+1})-(f(x_{k+1},y_k)-g(y_k))\\
    \leq &\frac{\rho L_y^2}{2}\|x_{k+1}-x_k\|^2+\frac{1}{\rho}\|y_{k+1}-y_k\|^2+\frac{1}{2\rho}\|y_k-y_{k-1}\|^2\\
    &+\frac{\gamma_{k-1}}{2}(\|y_{k+1}\|^2-\|y_k\|^2). 
\end{aligned} 
\end{equation}
On the other hand, following an argument similar to the proof of \cite[Lemma~5.2]{chenProximalGradientMethod2020}, we have
\begin{align*}
    f(x_{k+1}, y_k)+h(x_{k+1})\leq f(x_{k}, y_k)+h(x_{k})-(\frac{\beta_{k}}{\alpha_k}-b_2) \|\alpha_k v_{k}\|^2,
\end{align*}
where $b_2=M_2(L_h+G)+M_1^2L_x/2$.
By Proposition \ref{retr}, this and inequality \eqref{ineq:y-2} yield inequality \eqref{ineq:F-1}.
\end{proof}

 However, Lemma \ref{lem-desc-F-GA} alone does not yield an upper bound on a positively weighted sum of $\|v_k\|^2$ and $\|y_{k+1}-y_k\|^2$, which {provides} an upper bound for $\mathcal{G}^{\beta_k}(x_k,y_k)$ defined in \eqref{game-stationary measure}. To address this difficulty, we further refine the result in \eqref{ineq:F-1}. Specifically, in the next proposition, we derive a new inequality (see \eqref{ineq:H-1}) to better elucidate the relationship between $\|v_k\|^2$ and $\|y_{k+1}-y_k\|^2$. Furthermore, by using this new inequality, we construct a new value function (see \eqref{def-mF}) which involves $F_k$ and some additional terms associated with $\|y-y_{k-1}\|^2$ and $\|y\|^2$.

\begin{proposition}\label{prop:desc-F-GA}
 Let $x_k\in\mathcal{M},~y_{k-1} \in S$ be given. Define $v_k$ as in \eqref{update-v-GA}, and set $x_{k+1}:=R_{x_k}(\alpha_k v_k)$ for some $\alpha_k\in(0,1]$. Further, let $y_k:=\hat{y}_{k-1}(x_k)$ and $y_{k+1}:=\hat{y}_k(x_{k+1})$ with $\hat{y}_k$ defined in \eqref{y-update-GA}. Also denote \begin{align}
\mathcal{H}_{k}(y)&:=\left(\frac{4}{\rho^2\gamma_k} - \frac{4}{\rho}\right)\|y-y_{k-1}\|^2 + \frac{4}{\rho}\left(1-\frac{\gamma_{k-1}}{\gamma_k}\right)\|y\|^2,\label{def-hH}\\
{\mathcal{F}}_{k}(x,y)&:= F_k(x,y)+\frac{1}{2\rho}\|y-y_{k-1}\|^2 +(\frac{4}{\rho}\frac{\gamma_{k-1}}{\gamma_{k}}+\frac{\gamma_{k-1}}{2})\sigma_y^2+ \mathcal{H}_{k}(y).\label{def-mF}
\end{align}
If $\rho$ and $\gamma_k$ satisfy \eqref{range-rho,gamma}, then it holds that 
\begin{equation}\label{ineq:desc-F-2}
    \begin{aligned}
    &\mathcal{F}_{k+1}(x_{k+1},y_{k+1})-\mathcal{F}_{k}(x_k,y_k)\\\leq&-\big(\beta_k-(b_1+\frac{8L_y^2M_1^2}{\rho\gamma_k^2})\alpha_k\big)\alpha_k\|v_k\|^2-\frac{1}{10\rho}\|y_{k+1}-y_k\|^2.
\end{aligned}
\end{equation}
\end{proposition}

\begin{proof}

Let $f_k(x,y):=f(x,y)-\frac{\gamma_{k-1}}{2}\|y\|^2$. By \eqref{optimal-def-yk} and the concavity of $-g$, 
we have
\begin{equation*}
   \begin{aligned}
    \frac{1}{\rho}\langle w_k,y_{k+1}-y_k\rangle\leq&\langle\nabla_y f_{k+1}(x_{k+1},y_k)-\nabla_y f_{k}(x_k,y_{k-1}),y_{k+1}-y_k\rangle,
\end{aligned} 
\end{equation*}
where $w_k=y_{k+1}-2y_k+y_{k-1}$. 
Similar to the proof in Lemma \ref{lem-desc-F-GA}, combining this with \eqref{eq:nabla-inner}, \eqref{ineq:inner-1}, \eqref{ineq:inner-2} (replacing $\rho$ by $2/\gamma_k$), \eqref{ineq:inner-3}, \eqref{eq-gamma-y}, $\gamma_k\leq\gamma_{k-1}$, and the following identity
\begin{align*}
    &\frac{1}{\rho}\langle w_k,y_{k+1}-y_k\rangle=\frac{1}{2\rho}(\|y_{k+1}-y_k\|^2-\|y_k-y_{k-1}\|^2+\|w_k\|^2), 
\end{align*}
we deduce that 
\begin{equation}\label{ineq:H-1}
    \begin{aligned}
    \frac{1}{2\rho}\|y_{k+1}-y_k\|^2\leq&\frac{1}{2\rho}\|y_k-y_{k-1}\|^2+\frac{\gamma_k}{4}\|y_{k+1}-y_k\|^2-\frac{\gamma_{k}}{2}\|y_{k}-y_{k-1}\|^2\\&+\frac{L_y^2}{\gamma_k}\|x_{k+1}-x_k\|^2+\frac{\gamma_{k-1}-\gamma_k}{2}(\|y_{k+1}\|^2-\|y_k\|^2)\\
    \leq&\frac{1}{2\rho}\|y_k-y_{k-1}\|^2+\frac{\gamma_k}{4}\|y_{k+1}-y_k\|^2-\frac{\gamma_{k}}{2}\|y_{k}-y_{k-1}\|^2\\&+\frac{\alpha_k^2M_1^2L_y^2}{\gamma_k}\|v_k\|^2+\frac{\gamma_{k-1}-\gamma_k}{2}(\|y_{k+1}\|^2-\|y_k\|^2),
\end{aligned}
\end{equation}
where the last inequality comes from Proposition \ref{retr}.
Multiplying \eqref{ineq:H-1} by $\frac{8}{\rho\gamma_{k}}$ and then rearranging it appropriately, we obtain
\begin{equation}\label{ineq:H-2}
    \begin{aligned}
\mathcal{H}_{k+1}(y_{k+1})\leq&\mathcal{H}_{k}(y_k)+\frac{8\alpha_k^2M_1^2L_y^2}{\rho\gamma_k^2}\|v_k\|^2+\frac{4}{\rho}(\frac{\gamma_{k-1}}{\gamma_{k}}-\frac{\gamma_k}{\gamma_{k+1}})\|y_{k+1}\|^2\\
    &-\big(\frac{2}{\rho}-\frac{4}{\rho^2}(\frac{1}{\gamma_{k+1}}-\frac{1}{\gamma_{k}})\big)\|y_{k+1}-y_k\|^2.
\end{aligned}
\end{equation}
Since $\frac{\gamma_{k-1}}{\gamma_{k}}>\frac{\gamma_k}{\gamma_{k+1}}$ and for $\kappa>15$,  \[\frac{1}{\rho}(\frac{1}{\gamma_k}-\frac{1}{\gamma_{k-1}})=\frac{1}{2}\big((k+\kappa+2)^{1/4}-(k+\kappa+1)^{1/4}\big)<\frac{1}{2}\big(3^{1/4}-2^{1/4}\big)<\frac{1}{10},\]
by \eqref{ineq:F-1} and \eqref{ineq:H-2}, we obtain \eqref{ineq:desc-F-2}. 
\end{proof}

{While Lemma \ref{lem-desc-F-GA} and Proposition \ref{prop:desc-F-GA} follow a similar analytical framework to the corresponding results in \cite{xuUnifiedSingleloopAlternating2023}, our work introduces key innovations and essential distinctions. Specifically, we address nonsmooth nonconvex-concave minimax problems on Riemannian manifolds, in contrast to the smooth Euclidean setting considered in \cite{xuUnifiedSingleloopAlternating2023}. The proofs of Lemma \ref{lem-desc-F-GA} and Proposition \ref{prop:desc-F-GA} are novel in that they integrate Riemannian proximal gradient steps with retraction-based estimates and carefully handle nonsmooth subgradient terms—elements that are absent in \cite{xuUnifiedSingleloopAlternating2023}. These adaptations are not merely technical but necessary to ensure convergence in the Riemannian nonsmooth regime.}

Proposition \ref{prop:desc-F-GA} implies that we can ensure the sufficient descent of $\mathcal{F}_k(x_k,y_k)$ by choosing a small enough stepsize $\alpha_k \in (0,1]$. For possible acceleration, we further incorporate a backtracking line search procedure regarding inequality \eqref{ineq:desc-F-2} to determine $\alpha_k$. The MPGDA-PGA algorithm is formally presented in Algorithm \ref{alg:MPGDA-PGA}.
\begin{algorithm}[!ht]
 	\caption{Manifold Proximal Gradient Descent Ascent-Proximal Gradient Ascent (MPGDA-PGA) }
     \label{alg:MPGDA-PGA}
    \begin{algorithmic}[1]
        \REQUIRE $c_1\in(0,1)$, $\eta\in(0,1)$, $\kappa>15$, $0<\rho\leq\frac{1-2(\kappa+1)^{-1/4}}{L_y}$, $\gamma_k=\frac{2}{\rho(k+\kappa+2)^{1/4}}$, $0<\hat{l}_{\min}<\hat{l}_{\max}$, $x_0\in\mathcal{M}$, $y_{-1}\in S$.
        \STATE Set $y_0=\hat{y}_{-1}(x_0)$.
        \FOR{$k=0,1, \dots$}
        \STATE Choose $l_k\in [\hat{l}_{\min},\hat{l}_{\max}]$ and set $\beta_k=l_k/\gamma_k^2$.
        \STATE Compute $v_k$ by \eqref{update-v-GA}.
         \STATE Find the smallest non-negative integer $j$ such that $x_k^j:=R_{x_k}(\eta^j v_k)$ satisfies
         \begin{align}\label{eq:sufficient descent-GA}
            {\mathcal{F}}_{k+1}(x_k^j,\hat{y}_k(x_k^j)) \leq {\mathcal{F}}_{k}(x_k,y_k) - c_1\eta^j\beta_k\|v_k\|^2 - \frac{1}{10\rho}\|\hat{y}_k(x_k^j)-y_k\|^2,
         \end{align}
        \STATE Set $x_{k+1} = x_k^j, ~y_{k+1} = \hat{y}_{k}(x_{k+1})$.
        \ENDFOR
        \STATE \textbf{return} $(x_{k+1},y_{k+1})$
    \end{algorithmic}
\end{algorithm}

Before proceeding, we make some remarks on the iteration procedure of the MPGDA-PGA algorithm. First, the update of $y_{k+1}$ in MPGDA-PGA can be formulated as:
\begin{equation}
y_{k+1} = \mathop{\mathrm{argmin}}_{y \in S} \left\{ g(y) + \frac{1}{2\rho} \| y - \left( (1-\rho\gamma_k)y_k + \rho\nabla_y f(x_{k+1}, y_k) \right) \|^2 \right\},
\end{equation} which is a proximal problem of the type \eqref{proximal-g} and can be evaluated exactly and efficiently. Second, in practice, we use some special choices of $l_k \in [\hat{l}_{\min}, \hat{l}_{\max}]$ to enhance the efficiency of MPGDA-PGA in the experiments. Motivated by the Riemannian Barzilai-Borwein (BB) stepsize \cite{iannazzo2018riemannian,2013AWen}, we choose  $l_{k}$ as follows: 
\[l_k = \begin{cases}
    \min\left\{\max\left\{\hat{l}_{\min}, \gamma_k^2\left|\frac{\langle \Delta X_k, \Delta R_k\rangle}{\|\Delta X_k\|^2}\right|\right\},\hat{l}_{\max}\right\}, & \text{if } \Delta X_k \neq 0,\\
    \hat{l}_{\max}, & \text{if } \Delta X_k = 0,
\end{cases}\]
where $\Delta X_k = x_k - x_{k-1}$ and $\Delta R_k = \mathrm{grad} f(x_k, y_k) - \mathrm{grad} f(x_{k-1}, y_{k-1})$.

By Proposition \ref{prop:desc-F-GA}, the backtracking line search procedure in Algorithm \ref{alg:MPGDA-PGA} is well defined; i.e., it will terminate after a finite number of steps. We present this result in the following proposition.
\begin{proposition}\label{prop:line-search-GA}
For all $k\geq 0$, the condition \eqref{eq:sufficient descent-GA} in Algorithm \ref{alg:MPGDA-PGA} is satisfied within at most $J_2$ backtracking steps, where
    \[J_2:=\max\left\{\left\lceil \log_{\eta}\left(\frac{\rho(1-c_1)\hat{l}_{\min}}{\rho b_1\gamma_{-1}^2+8L_y^2M_1^2}\right)\right\rceil,0\right\}\]
    with $b_1$ being defined in Lemma \ref{lem-desc-F-GA}.
\end{proposition}

Below we establish the iteration complexity for Algorithm \ref{alg:MPGDA-PGA}. To this end, we first define two constants $\overline{\mathcal{F}}$ and $\underline{\mathcal{F}}$ as 
\begin{align*}
    \overline{\mathcal{F}}&:=\sup\{{F(x,y)+(\frac{6}{\rho}+\frac{16}{\rho^2\gamma_0}+\frac{4\gamma_{-1}}{\rho\gamma_0}+\frac{\gamma_{-1}}{2})\sigma_y^2\mid x\in\mathcal{M},y\in S}\},\\ \underline{\mathcal{F}}&:=\inf\{F(x,y)-(\frac{\gamma_{-1}}{2}+\frac{24}{\rho})\sigma_y^2\mid x\in\mathcal{M},y\in S\}.
\end{align*} By direct calculation, we have  $\underline{\mathcal{F}}\leq \mathcal{F}_{k}(x,y)$ for all $k\geq0$ and $\mathcal{F}_0(x,y)\leq \overline{\mathcal{F}}$
for any  point $(x,y)\in\mathcal{M}\times(S)$.

\begin{theorem}\label{the-GA-complexity}
     Given $\varepsilon>0$, let $\{(x_{k},y_{k})\}$ be generated by Algorithm \ref{alg:MPGDA-PGA} and $T(\varepsilon)$ be defined in \eqref{complex-measure}.  Then, it holds that 
     \begin{align*}
         T(\varepsilon)\leq \frac{\big((15+4\kappa)^{1/4}(2(\overline{\mathcal{F}}-\underline{\mathcal{F}})/C_1)^{1/2}(\rho L_y+1)+2\sigma_y/\rho\big)^{4}}{\varepsilon^{4}},
     \end{align*}
 where \[C_1=\min\{\frac{4c_1\eta^{J_2}}{\hat{l}_{\max}\rho^2},\frac{\rho(\kappa+2)^{1/2}}{10}\}.\] Here, the constants 
$c_1,~\eta,~\kappa,~\hat{l}_{\max}$, and $\rho$ are parameters of Algorithm \ref{alg:MPGDA-PGA}, and $J_2$ is defined in Proposition \ref{prop:line-search-GA}.
\end{theorem}
\begin{proof}
   Summing the inequality \eqref{eq:sufficient descent-GA} over $k=0, 1, \dots, K$, it holds that 
\begin{align} \label{ineq:cal-F} \sum_{k=0}^{K}c_1\eta^{J_2}\beta_k\|v_{k}\|^2+\frac{1}{10\rho}\|y_{k+1}-y_k\|^2\leq \overline{\mathcal{F}}-\underline{\mathcal{F}},
\end{align} 
 which follows from Proposition \ref{prop:line-search-GA} and the definition of $\overline{\mathcal{F}}$ and $\underline{\mathcal{F}}$. Combining \eqref{ineq:cal-F} and the definition of $\beta_k$, we have 
{
\begin{align}\label{ineq:beta v}
   (\sum_{k=0}^K \frac{C_1}{(k+\kappa+2)^{1/2}})\min_{0\leq k\leq K} (\|\beta_k v_k\|^2+\|\frac{1}{\rho}(y_{k+1}-y_{k})\|^2)&\leq\overline{\mathcal{F}}-\underline{\mathcal{F}},
\end{align}}
where $C_1=\min\{\frac{4c_1\eta^{J_2}}{\hat{l}_{\max}\rho^2},\frac{\rho(\kappa+2)^{1/2}}{10}\}$.
 Note that 
\begin{align*}
    \sum_{k=0}^{K} \frac{1}{(k+\kappa+2)^{1/2}} &\geq \int_{1}^{K+1} \frac{1}{(x+\kappa+2)^{1/2}} dx\\
    &=\frac{1}{2}((K+\kappa+3)^{1/2}-(\kappa+3)^{1/2}),
\end{align*}  
this and inequality \eqref{ineq:beta v} yield that
{\begin{align}\label{ineq:max beta v,y-y}
    \min_{0\leq k\leq K } \|\beta_k v_k\|^2+\|\frac{1}{\rho}(y_{k+1}-y_{k})\|^2&\leq \frac{2(\overline{\mathcal{F}}-\underline{\mathcal{F}})}{C_1\big((K+\kappa+3)^{1/2}-(\kappa+3)^{1/2}\big )}.
\end{align}}
On the other hand, by the optimal condition of definition of $y_{k}$, for $k\geq 2$, we have
{\begin{align*}
    \mathrm{dist}(0,\nabla_y f(x_k,y_k)-\partial g(y_k)-N_{y_k}S )\leq (L_y+\frac{1}{\rho})\|y_k-y_{k-1}\|+\gamma_{k-1}\|y_k\|.
\end{align*}}
 Combining this bound with the definition of $\gamma_{k-1}$ and inequality \eqref{ineq:max beta v,y-y}, we obtain that
     \begin{align}\label{ineq-measure_GA-prf}
    \min_{0\leq k\leq K}\mathcal{G}^{\beta_k}(x_k,y_k)\leq\frac{(2(\overline{\mathcal{F}}-\underline{\mathcal{F}})/C_1)^{1/2}(\rho L_y+1)}{((K+\kappa+3)^{1/2}-(\kappa+3)^{1/2})^{1/2}}+\frac{2\sigma_y}{\rho(K+2)^{1/4}}.
\end{align}
 By the fact that $(K+\kappa+3)^{1/2}-(\kappa+3)^{1/2}\geq \frac{K^{1/2}}{(15+4\kappa)^{1/2}}$ for $\kappa>15$ and  any $K\geq 1$, we deduce from \eqref{ineq-measure_GA-prf} that 
 \begin{align*}
    \min_{0\leq k\leq K}\mathcal{G}^{\beta_k}(x_k,y_k)\leq\frac{(15+4\kappa)^{1/4}(2(\overline{\mathcal{F}}-\underline{\mathcal{F}})/C_1)^{1/2}(\rho L_y+1)+2\sigma_y/\rho}{K^{1/4}},
    \end{align*}
    and the proof is finished.
\end{proof}

Theorem \ref{the-GA-complexity} demonstrates that the proposed MPGDA-PGA algorithm can find an $\varepsilon$-game-stationary point for problem \eqref{Prob} within $\mathcal{O}(\varepsilon^{-4})$ outer iterations. In view of Proposition \ref{epsilon-staionary}, this also indicates  that the outer iteration complexity of the MPGDA-PGA algorithm for returning an $\varepsilon$-optimization-stationary point is $\mathcal{O}(\varepsilon^{-4})$.

\section{Numerical Experiments}\label{sec:6}
In this section, we illustrate the proposed algorithms via three numerical examples. We begin with an analytic example of problem (1), where the objective function is not linear with respect to $y$. This example is used to illustrate the differences in behavior between the MPGDA-PA algorithm and the MPGDA-PGA algorithm. Then, we examine the performance of the MPGDA-PA algorithm for solving the FSPCA problem \eqref{FSPCA} and the SSC problem \eqref{SSC3} respectively. Since these two problems are nonconvex-linear, the two proposed algorithms have comparable computational cost at each iteration, but the MPGDA-PA algorithm enjoys better outer iteration complexity. Hence, it is more reasonable to apply MPGDA-PA rather than MPGDA-PGA algorithm for solving problems \eqref{FSPCA} and \eqref{SSC3}.
 All experiments are implemented in MATLAB R2025a and conducted on a standard PC with 3.40GHz Intel(R) Core(TM) i5-7500 CPU and 24GB of RAM.

 We first specify some implementation details of the proposed MPGDA-PA algorithm. Throughout the tests, we set $c_1=10^{-4}$, $\eta=0.1$, $l_{\min}=10^{-16}$, $l_{\max}=10^{16}$, $\delta_0=10^{10}$, $\tau_1=0.999$, and $\tau_2=0.9$. The parameters $\gamma_0$, $\xi_0>0$, $\theta>1$ and $\{T_k\}$ vary across the three examples. Since $h$ is the $\ell_1$-norm in both problems, the semismooth Newton method developed in \cite[Section~4.2]{chenProximalGradientMethod2020} is used to solve subproblem \eqref{v ki}. We utilize the QR factorization as the retraction on the Stiefel manifold \cite{2008Optimization}. 

\subsection{Experiments on an Analytic Minimax Problem}
Consider the following nonconvex-concave minimax problem on $\mathbb{R}^2\times\mathbb{R}$:
\begin{align}\label{prob-analytic}
\min_{x_1^2 + x_2^2 = 1} \max_{0.3 \leq y \leq 1} -0.01x_1^3 y - y \ln(y),
\end{align}
which is a special case of problem \eqref{Prob} with $\mathcal{M}$ being the Stiefel manifold $\mathrm{St}(2,1)$, $S$ being the interval $[0.3, 1]$, and $f(x,y) = -0.01x_1^3 y - y \ln(y)$. A direct verification shows that $(1, 0)$ is the unique minimizer to the value function of \eqref{prob-analytic} and $(x^*,y^*):=(1,0,e^{-1.01})$ is a game-stationary point. Now we apply the two proposed algorithms to solve  \eqref{prob-analytic}. For the MPGDA-PA algorithm, we set $\gamma_0 =0.005$, $\xi_0 = 1$, $\theta = 1.5$, and $T_k = 1$. We use the bisection method to solve the corresponding $y$-subproblem \eqref{y-bar}, which is implemented via the \texttt{fminbnd} command in MATLAB. For MPGDA-PGA, we set $c_1=10^{-4}$, $\hat{l}_{\min}=10^{-16}$, $\hat{l}_{\max}=10^8$, $\rho = 0.2$, $\kappa = 10^{16}$, and $\eta = 0.5$. Both algorithms are initialized from the same starting point $(0.8,0.6,0.3) \in \mathrm{St}(2,1) \times [0.3, 1]$. We use the Euclidean distance between the iterates $(x_k, y_k)$ and the optimal point $(x^*, y^*)$, denoted by $D_k$, to measure the accuracy of the iterates generated by them. 

\begin{table}[h]
\centering
\caption{Comparison of iteration numbers and CPU time for reaching different levels of distance}
\label{tab:analytic}
\begin{tabular}{cccccc}
\toprule
Level & \multicolumn{2}{c}{Iter.} & \multicolumn{2}{c}{Time} \\
\cmidrule(r){2-3} \cmidrule(r){4-5}
 & MPGDA-PA & MPGDA-PGA & MPGDA-PA & MPGDA-PGA \\
\midrule
$1\times10^{-2}$ & 17 & 918 & 0.038659 & 0.023850 \\
$1\times10^{-3}$ & 19 & 2100 & 0.044875 & 0.054841 \\
$3\times10^{-4}$ & 21 & 2767 & 0.058287 & 0.073542 \\
$2\times10^{-4}$ & 38 & 3067 & 0.114491 & 0.084889 \\
$1.5\times10^{-4}$ & 88 & 3455 & 0.167304 & 0.118553 \\
\bottomrule
\end{tabular}
\end{table}

Table \ref{tab:analytic} presents the first outer iteration number ( Iter.) required by the proposed algorithms to reach different levels of distance, and the corresponding CPU time in seconds (Time). Moreover, Figures \eqref{fig:convergence_iters} and \eqref{fig:convergence_time} show  $D_k$ (in logarithmic scale) versus iteration number $k$ and $D_k$ (in logarithmic scale) versus CPU time in seconds respectively. We observe that the iterates generated by both proposed algorithms converge to the optimal solution $(x^*, y^*)$. Although the MPGDA-PA algorithm requires {solving} a nonlinear and strongly concave subproblem with respect to $y$ in each iteration, which results in a higher computational cost per iteration, it converges in substantially fewer outer iterations. In contrast, MPGDA-PGA performs simpler iterations at a much lower cost per iteration, but requires much more iterations to achieve comparable accuracy. Consequently, the overall CPU times of both algorithms are similar across different levels of accuracy, demonstrating a consistent trade-off between per-iteration cost and convergence speed for solving problem \eqref{prob-analytic}.
\begin{figure}[htbp]
    \centering
    \begin{subfigure}[b]{0.48\textwidth}
        \centering
        \includegraphics[width=\linewidth]{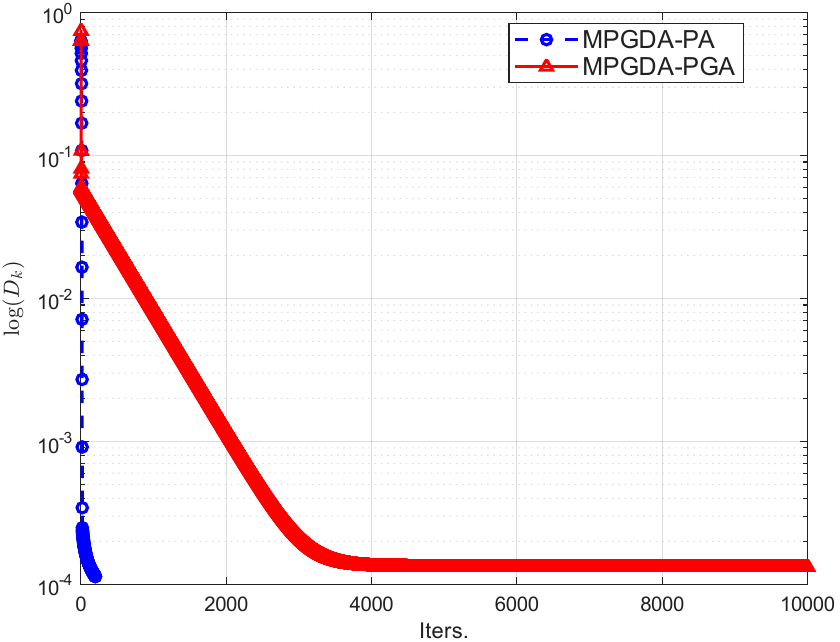}
        \caption{$D_k$ (in log scale) vs Iters}
        \label{fig:convergence_iters}
    \end{subfigure}
    \hfill
    \begin{subfigure}[b]{0.48\textwidth}
        \centering
        \includegraphics[width=\linewidth]{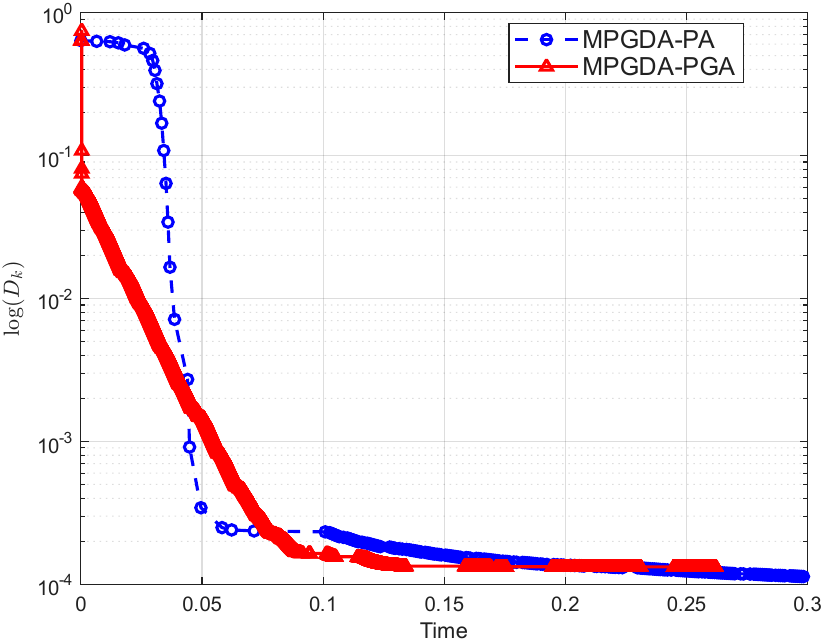}
        \caption{$D_k$ (in log scale) vs Time}
        \label{fig:convergence_time}
    \end{subfigure}
    \caption{Plots of $D_k$ (in log scale) for the proposed algorithms.}
    \label{fig:convergence_comparison}
\end{figure}

\subsection{Experiments on FSPCA}\label{FSPCA-exp}  
In this subsection, we compare the proposed MPGDA-PA algorithm with the constraint relaxed majorization-minimization (CRMM) algorithm \cite[Algorithm~3]{babuFairPrincipalComponent2023} for solving FSPCA problem \eqref{FSPCA}. The CRMM algorithm first relaxes the semi-orthogonality constraint $X^T X = I$ into a linear matrix inequality constraint and then applies some majorization-minimization iteration procedure, which requires solving a quadratic program at each iteration.

We perform tests on synthetic datasets and a real dataset where the samples belong to two groups. The synthetic datasets are generated following similar settings to \cite[Section~5.1]{2016SparseGeneralizedEigenvalueProblem}. In detail, for the $i$-th group, we randomly generate $m_i =200$ samples with $d=40$ features according to the Gaussian distribution with mean $\mu^i \in \mathbb{R}^d$ and covariance $\Sigma \in \mathbb{R}^{d \times d}$. We set
$\mu_j^1 = 0,\ \mu_j^2 = 1/3$, for $j \in \{2,4,\cdots,200\}$, and $\mu_j^i = 0$
 otherwise. Meanwhile, the covariance matrix $\Sigma$ is chosen as a block diagonal matrix with five $8 \times 8$ blocks, where the $(j,j')$-th entry of each block is $0.8^{|j-j'|}$. For the real-world data, we use the default credit dataset \cite{yeh2009creditdata}  and preprocess it as done in \cite{samadiPriceFairPCA}. We set the regularization parameter $\mu=0.1$ in our tests. Furthermore, the CRMM algorithm is terminated if the relative error between the solutions obtained over two successive iterations is less than $10^{-6}$. For the MPGDA-PA algorithm, we set $\gamma_0=10^{-6}$, $\xi_0=4\sqrt{r}*10^{4}$, $\theta=1.5$, $T_k\equiv15$, and we terminate the algorithm if it satisfies:
\begin{align}\label{terminate}
    \mathcal{G}^{\beta_{k,0}}(x_k, y_k) < \varepsilon \text{ with } \varepsilon=10^{-6}.
\end{align}
 Both compared  algorithms start from the same random initial points on $\mathrm{St}(d,r)$ and the maximum iteration number of them are set as 1000.

 We present the average results over 50 runs on different synthetic datasets and  results for the default credit dataset in Table \ref{tab:FSPCA}. We report the objective value (Obj), CPU time in seconds (Time) and the number of iterations (Iter). It is observed that the MPGDA-PA algorithm achieves similar objective values to the CRMM algorithm while requiring much less computational time. Specifically, our proposed algorithm is on average more than 20 times faster than the CRMM algorithm.

\begin{table}[htbp]
\caption{Results for FSPCA}
\label{tab:FSPCA}
\centering
\setlength{\tabcolsep}{4pt}
\begin{tabular}{
c
c
S[table-format=-2.3]
S[table-format=2.3]
S[table-format=3.0]
S[table-format=-2.3]
S[table-format=1.3]
S[table-format=3.0]
}
\toprule
 \multicolumn{2}{c}{} & \multicolumn{3}{c}{CRMM} & \multicolumn{3}{c}{MPGDA-PA} \\
\cmidrule(lr){3-5} \cmidrule(lr){6-8}
{Dataset} & {r} & {Obj.} & {Time} & {Iter.} & {Obj.} & {Time} & {Iter.} \\
\midrule
\multirow{4}{*}{Synthetic}
& 2 & -9.802 & 0.122 & 204 & -9.820 & \textbf{0.005} & 167 \\
& 3 & -14.399 & 0.131 & 165 & -14.401 & \textbf{0.007} & 189 \\
& 4 & -18.755 & 0.175 & 148 & -18.776 & \textbf{0.008} & 216 \\
& 5 & -22.840 & 0.089 & 79 & -22.867 & \textbf{0.007} & 216 \\
\midrule
\multirow{4}{*}{Credit}
& 2 & -9.531 & 4.305 & 93 & -9.742 & \textbf{0.276} & 124 \\
& 3 & -10.557 & 14.599 & 272 & -10.802 & \textbf{0.138} & 91 \\
& 4 & -11.637 & 6.478 & 143 & -11.923 & \textbf{0.013} & 5 \\
& 5 & -12.572 & 14.332 & 274 & -12.724 & \textbf{0.013} & 4 \\
\bottomrule
\end{tabular}
\end{table}

\subsection{Experiments on SSC}\label{SSC-exp}
In this subsection, we compare the proposed MPGDA-PA method with the RADA-RGD and RADA-PGD algorithms in  \cite{xuRiemannianAlternatingDescent2024} for solving the SSC problem. Note that in \cite{xuRiemannianAlternatingDescent2024}, the original SSC problem \eqref{SSC1} is reformulated into the following smooth minimax problem on the Grassmann manifold so that the RADA-PGD and RADA-RGD algorithms can be directly applied:
\begin{align}\label{Xu:SSC}
    \min_{Z\in \mathrm{Gr}(N,p)}\max_{\|Y\|_{\infty}\leq \mu}\{\langle L,Z\rangle+\langle Y,Z\rangle\}.
\end{align} 
Here, $\mathrm{Gr}(N,p)=\{XX^{\top}\in\mathbb{R}^{N\times N}\mid X\in\mathrm{St}(N,p) \}$ denotes the Grassmann manifold \cite{AGrassmannmanifoldhandbook2024A}. The code of these two algorithms for solving \eqref{Xu:SSC} is downloaded from:
    \url{https://github.com/XuMeng00124/RADAopt}. 
In the tests, we generate synthetic datasets as in \cite{2016ConvexSparseSpectralClustering}, where the data points $\{a_i\}_{i=1}^N$ are independently drawn from the standard Gaussian distribution and $W_{ij}=|\langle a_i,a_j\rangle|$. For the two RADA algorithms, we adopt the parameters as suggested in \cite[Section~6.3]{xuRiemannianAlternatingDescent2024} except the number of inner iterations ($T_k$). Specifically, they are terminated once an $\varepsilon$-game-stationary point of problem \eqref{Xu:SSC}  is found with $\varepsilon=10^{-4}$ and the initial point is set as $X_0X_0^{\top}$, where $X_0$ consists of $p$ eigenvectors associated with the $p$ smallest eigenvalues of $L$. For the MPGDA-PA algorithm, we set $\gamma_0=10^{-5}$, $\xi_0=\sqrt{p}N^{2}$, and $\theta=2$.  Moreover, we choose the aforementioned $X_0$ as the initial point of the MPGDA-PA algorithm and terminate the method if it satisfies \eqref{terminate} with $\varepsilon=10^{-4}$. The maximum number of iterations for all the compared algorithms is set as 1000. 
  
In the experiments, we generate 50 synthetic datasets with $N=200$ data points. To fine tune the parameters $T_k$ involved, we test all the compared algorithms with $T_k \in \{1,2,3,4,5\}$ for solving the SSC problem with various choices of $p$ and $\mu$. {For \(T_k=1\), both RADA‑RGD and MPGDA‑PA did not satisfy the termination criterion within the maximum of 1000 outer iterations; consequently, their reported objective values are still far from the optimum, indicating that more iterations are needed to reach a stationary point. For \(2\leq T_k\leq4\), the objective values of all three algorithms are very close. These results show that increasing the number of inner (proximal) gradient steps appropriately helps reduce the number of outer iterations and leads to comparable solution quality. The results also reveal that the RADA-PGD, RADA-RGD and MPGDA-PA algorithms consistently achieve optimal performance with $T_k = 1$, $2$, and $3$, respectively.} Due to the limitation of space, we only report one representative result in Table \ref{tab:SSC-1}, which shows the performance evaluation for different choices of $T_k$ on the SSC problem with $(p, \mu) = (5, 1.0)$. Also, the computational results for the compared algorithms with their respective optimal $T_k$ are summarized in Table \ref{tab:SSC-2}. Both tables present the {(primal)} objective function value of problem \eqref{SSC1} (Obj.), the CPU time in seconds (Time), and number of outer iterations (Iter.) averaged over the 50 datasets. We observe that all the compared algorithms attain comparable objective function values for different numbers of groups and regularization parameters $\mu$. In addition, the proposed MPGDA-PA method consistently outperforms the two RADA algorithms in terms of CPU time, which again validates its computational efficiency. The superior performance of the proposed MPGDA-PA algorithm may be attributed to the fact that it is able to directly address the nonsmooth $\ell_1$ norm in the SSC problem, while the RADA algorithms can only tackle the converted smooth minimax formulation of SSC.

\begin{table}[htbp]
\centering
\caption{Performance evaluation for different choices of $T_k$ on the SSC problem with $(p,\mu)=(5,1.0)$}
\label{tab:SSC-1}
\setlength{\tabcolsep}{4pt}
\begin{tabular}{
 l 
 S[table-format=2.3] 
 S[table-format=1.3] 
 S[table-format=3.0] 
 S[table-format=2.3]
 S[table-format=1.3]
 S[table-format=4.0]
 S[table-format=2.3]
 S[table-format=1.3]
 S[table-format=3.0]
}
\toprule
\multirow{2}{*} & 

\multicolumn{3}{c}{{RADA-PGD}} & 
\multicolumn{3}{c}{{RADA-RGD}} & 
\multicolumn{3}{c}{{MPGDA-PA}} \\
\cmidrule(lr){2-4} \cmidrule(lr){5-7} \cmidrule(lr){8-10}
{$T_k$} & {Obj.} & {Time} &  {Iter.} &  {Obj.} &  {Time} &  {Iter.} &  {Obj.} &  {Time} &  {Iter.} \\
\midrule
1 & 9.971 & \textbf{0.578} & 300 & 17.590 & 1.787 & 1000 & 22.320 & 2.063 & 1000 \\
2 & 9.971 & 1.148 & 300 & 10.070 & \textbf{0.873} & 238 & 10.029 & 2.080 & 848 \\
3 & 9.971 & 1.629 & 297 & 9.967 & 1.170 & 217 & 9.968 & \textbf{0.256} & 97 \\
4 & 9.971 & 2.198 & 301 & 9.967 & 1.830 & 254 & 9.968 & {0.298} & 88 \\
5 & 9.971 & 2.516 & 300 & 64.748 & 2.742 & 289 & 9.970 & {0.311} & 84 \\
\bottomrule
\end{tabular}
\end{table}

\begin{table}[htbp]
\centering
\caption{Results for SSC}
\label{tab:SSC-2}
\setlength{\tabcolsep}{3.5pt}
\begin{tabular}{
 c
 S[table-format=2.3]
 S[table-format=1.3]
 S[table-format=3.0]
 S[table-format=2.3]
 S[table-format=1.3]
 S[table-format=3.0]
 S[table-format=2.3]
 S[table-format=1.3]
 S[table-format=2.0]
}
\toprule
\multirow{2}{*}{\makecell{Parameters \\($p$, $\mu$)}} &
\multicolumn{3}{c}{{RADA-PGD}} &
\multicolumn{3}{c}{{RADA-RGD}} &
\multicolumn{3}{c}{{MPGDA-PA}} \\
\cmidrule(lr){2-4} \cmidrule(lr){5-7} \cmidrule(lr){8-10}
& {Obj.} & {Time} &  {Iter.} &  {Obj.} &  {Time} &  {Iter.} &  {Obj.} &  {Time} &  {Iter.} \\
\midrule
($2$, $0.1$)  &  2.187 & 0.408 & 265 &  2.187 & 0.451 & 154 &  2.187 & \textbf{0.106} &  66 \\
($4$, $0.1$)  &  4.374 & 0.410 & 259 &  4.383 & 0.517 & 180 &  4.373 & \textbf{0.101} &  64 \\
($6$, $0.1$)  &  6.560 & 0.400 & 224 &  6.569 & 0.643 & 223 &  6.560 & \textbf{0.103} &  63 \\
($8$, $0.1$)  &  8.747 & 0.431 & 198 &  8.753 & 0.765 & 264 &  8.747 & \textbf{0.109} &  66 \\
($10$, $0.1$) & 10.934 & 0.474 & 180 & 10.956 & 0.943 & 277 & 10.934 & \textbf{0.113} &  64 \\
\midrule
($5$, $0.1$)  &  5.467 & 0.416 & 244 &  5.494 & 0.719 & 240 &  5.467 & \textbf{0.110} &  65 \\
($5$, $0.2$)  &  5.967 & 0.437 & 263 &  5.977 & 0.636 & 217 &  5.967 & \textbf{0.117} &  71 \\
($5$, $0.5$)  &  7.469 & 0.467 & 286 &  7.489 & 0.672 & 229 &  7.467 & \textbf{0.155} &  86 \\
($5$, $1.0$)  &  9.971 & 0.490 & 300 & 10.070 & 0.702 & 238 &  9.968 & \textbf{0.189} &  97 \\
\bottomrule
\end{tabular}
\end{table}

~~

 \noindent \textbf{Funding} Qia Li is supported by the National Natural Science Foundation of China under
 grant 12471098, and by the Guangdong Province Key Laboratory of Computational National Science
 at Sun Yat-sen University Grant 2020B1212060032.

~~

 \noindent\textbf{Data Availability} The default credit dataset used in Section \ref{FSPCA-exp} was downloaded from the website \url{https://archive.ics.uci.edu/dataset/350/default+of+credit+card+clients}.

%
\section*{Declarations}
\textbf{Conflict of interest}~~The authors declare that they have no conflict of interest.

\bibliographystyle{spmpsci}      
\bibliography{reference_v2}

@article{aybat2025retraction,
  title={A Retraction-free Method for Nonsmooth Minimax Optimization over a Compact Manifold},
  author={Aybat, Necdet Serhat and Hu, Jiang and Deng, Zhanwang},
  journal={ arXiv:2510.22065},
  year={2025}
}

@article{li2026smoothing,
  title={Smoothing Meets Perturbation: Unified and Tight Analysis for Nonconvex-Concave Minimax Optimization},
  author={Li, Jiajin and Nagarajan, Mahesh and Pan, Siyu and Zhang, Nanxi},
  journal={arXiv:2602.14185},
  year={2026}
}

@article{li2024proximal,
  title={Proximal methods for structured nonsmooth optimization over Riemannian submanifolds},
  author={Li, Qia and Zhang, Na  and Feng, Junyu and Yan, Hanwei},
  journal={arXiv:2411.15776},
  year={2024}
}

@article{perez2019subdifferential,
  title={Formulae for the conjugate and the subdifferential of the supremum function},
  author={P{\'e}rez-Aros, Pedro},
  journal={Journal of Optimization Theory and Applications},
  volume={180},
  number={2},
  pages={397--427},
  year={2019},
  publisher={Springer}
}

@article{iannazzo2018riemannian,
  title={The Riemannian Barzilai--Borwein method with nonmonotone line search and the matrix geometric mean computation},
  author={Iannazzo, Bruno and Porcelli, Margherita},
  journal={IMA Journal of Numerical Analysis},
  volume={38},
  number={1},
  pages={495--517},
  year={2018},
  publisher={Oxford University Press}
}

@article{2013AWen,
  title={A feasible method for optimization with orthogonality constraints},
  author={Wen, Zaiwen and Yin, Wotao},
  journal={Mathematical Programming},
  volume={142},
  number={1},
  pages={397--434},
  year={2013},
  publisher={Springer}
}

@article{yeh2009creditdata,
  title={The comparisons of data mining techniques for the predictive accuracy of probability of default of credit card clients},
  author={Yeh, I-Cheng and Lien, Che-hui},
  journal={Expert systems with applications},
  volume={36},
  number={2},
  pages={2473--2480},
  year={2009},
  publisher={Elsevier}
}

@article{2016SparseGeneralizedEigenvalueProblem,
  title={Sparse generalized eigenvalue problem: Optimal statistical rates via truncated rayleigh flow},
  author={Tan, Kean Ming and Wang, Zhaoran and Liu, Han and Zhang, Tong},
  journal={Journal of the Royal Statistical Society Series B: Statistical Methodology},
  volume={80},
  number={5},
  pages={1057--1086},
  year={2018},
  publisher={Oxford University Press}
}

@article{AGrassmannmanifoldhandbook2024A,
  title={A Grassmann manifold handbook: Basic geometry and computational aspects},
  author={Bendokat, Thomas and Zimmermann, Ralf and Absil, P-A},
  journal={Advances in Computational Mathematics},
  volume={50},
  number={1},
  pages={6},
  year={2024},
  publisher={Springer}
}

@article{BriefIntroductiontoManifoldOptimization,
  title={A brief introduction to manifold optimization},
  author={Hu, Jiang and Liu, Xin and Wen, Zai-Wen and Yuan, Ya-Xiang},
  journal={Journal of the Operations Research Society of China},
  volume={8},
  pages={199--248},
  year={2020},
  publisher={Springer}
}

@article{2016ConvexSparseSpectralClustering,
  title={Convex Sparse Spectral Clustering: Single-View to Multi-View},
  author={ Lu, Canyi  and  Yan, Shuicheng  and  Lin, Zhouchen },
  journal={IEEE Transactions on Image Processing},
  volume={25},
  number={6},
  pages={2833-2843},
  year={2016},
}

@article{babuFairPrincipalComponent2023,
  title={{Fair principal component analysis (PCA):minorization-maximization algorithms for Fair PCA, Fair Robust PCA and Fair Sparse PCA}},
  author={Babu, Prabhu and Stoica, Petre},
  journal={arXiv:2305.05963},
  year={2023}
}

@article{chenProximalGradientMethod2020,
  title = {Proximal {{Gradient Method}} for {{Nonsmooth Optimization}} over the {{Stiefel Manifold}}},
  author = {Chen, Shixiang and Ma, Shiqian and {Man-Cho So}, Anthony and Zhang, Tong},
  year = {2020},
  month = jan,
  journal = {SIAM Journal on Optimization},
  volume = {30},
  number = {1},
  pages = {210--239},
  issn = {1052-6234, 1095-7189},
}

@book{1983Optimization,
author = {Clarke, Frank H.},
title = {Optimization and Nonsmooth Analysis},
publisher = {Society for Industrial and Applied Mathematics},
address = {Philadelphia, PA},
year = {1990},
}

@article{danskinTheoryMaxMinApplications1966,
  title = {The {{Theory}} of {{Max-Min}}, with {{Applications}}},
  author = {Danskin, John M.},
  year = {1966},
  month = jul,
  journal = {SIAM Journal on Applied Mathematics},
  volume = {14},
  number = {4},
  pages = {641--664},
  issn = {0036-1399, 1095-712X},
}

@ARTICLE{huangGradientDescentAscent2023,
  author={Huang, Feihu and Gao, Shangqian},
  journal={IEEE Transactions on Pattern Analysis and Machine Intelligence}, 
  title={Gradient Descent Ascent for Minimax Problems on Riemannian Manifolds}, 
  year={2023},
  volume={45},
  number={7},
  pages={8466-8476},
}

@inproceedings{nouiehedSolvingClassNonConvex,
 author = {Nouiehed, Maher and Sanjabi, Maziar and Huang, Tianjian and Lee, Jason D and Razaviyayn, Meisam},
 booktitle = {Advances in Neural Information Processing Systems},
 editor = {H. Wallach and H. Larochelle and A. Beygelzimer and F. d\textquotesingle Alch\'{e}-Buc and E. Fox and R. Garnett},
 pages = {},
 title = {Solving a Class of Non-Convex Min-Max Games Using Iterative First Order Methods},
volume = {32},
 year = {2019},
}

@inproceedings{samadiPriceFairPCA,
 author = {Samadi, Samira and Tantipongpipat, Uthaipon and Morgenstern, Jamie H and Singh, Mohit and Vempala, Santosh},
 booktitle = {Advances in Neural Information Processing Systems},
 editor = {S. Bengio and H. Wallach and H. Larochelle and K. Grauman and N. Cesa-Bianchi and R. Garnett},
 pages = {},
 title = {The Price of Fair PCA: One Extra dimension},
 volume = {31},
 year = {2018}
}

@inproceedings{thekumparampilEfficientAlgorithmsSmooth,
 author = {Thekumparampil, Kiran K and Jain, Prateek and Netrapalli, Praneeth and Oh, Sewoong},
 booktitle = {Advances in Neural Information Processing Systems},
 editor = {H. Wallach and H. Larochelle and A. Beygelzimer and F. d\textquotesingle Alch\'{e}-Buc and E. Fox and R. Garnett},
 pages = {},
 title = {Efficient Algorithms for Smooth Minimax Optimization},
 volume = {32},
 year = {2019}
}

@inproceedings{xuEfficientAlternatingRiemannian2024,
  title = {An {{Efficient Alternating Riemannian}}/{{Projected Gradient Descent Ascent Algorithm}} for {{Fair Principal Component Analysis}}},
  booktitle = {{{ICASSP}} 2024 - 2024 {{IEEE International Conference}} on {{Acoustics}}, {{Speech}} and {{Signal Processing}} ({{ICASSP}})},
  author = {Xu, Meng and Jiang, Bo and Pu, Wenqiang and Liu, Ya-Feng and So, Anthony Man-Cho},
  year = {2024},
  month = apr,
  pages = {7195--7199},
}

@article{Hotelling1933Analysis,
  title={Analysis of a complex of statistical variables in principal components},
  author={Hotelling and H.},
  journal={Journal of Educational Psychology},
  volume={24},
  number={7},
  pages={498-520},
  year={1933},
}

@book{2008Optimization,
  title={Optimization Algorithms on Matrix Manifolds},
  author={ Absil, Pierre Antoine  and  Mahony, Robert  and  Sepulchre, Rodolphe },
  publisher={Princeton University Press},
  year={2008},
 address = {Princeton},
}

@Book{boumal2023intromanifolds,
  title     = {An introduction to optimization on smooth manifolds},
  author    = {Boumal, Nicolas},
  publisher = {Cambridge University Press},
  year      = {2023},
 address={Cambridge},
}

@book{Beck2017First,
author = {Beck, Amir},
title = {First-Order Methods in Optimization},
publisher = {Society for Industrial and Applied Mathematics},
address = {Philadelphia, PA},
year = {2017},
}

@article{boumalGlobalRatesConvergence2019,
  title = {Global Rates of Convergence for Nonconvex Optimization on Manifolds},
  author = {Boumal, Nicolas and Absil, P-A and Cartis, Coralia},
  year = {2019},
  month = jan,
  journal = {IMA Journal of Numerical Analysis},
  volume = {39},
  number = {1},
  pages = {1--33},
  issn = {0272-4979, 1464-3642},
}

@article{luHybridBlockSuccessive2020,
  title = {Hybrid {{Block Successive Approximation}} for {{One-Sided Non-Convex Min-Max Problems}}: {{Algorithms}} and {{Applications}}},
  shorttitle = {Hybrid {{Block Successive Approximation}} for {{One-Sided Non-Convex Min-Max Problems}}},
  author = {Lu, Songtao and Tsaknakis, Ioannis and Hong, Mingyi and Chen, Yongxin},
  year = {2020},
  journal = {IEEE Transactions on Signal Processing},
  volume = {68},
  pages = {3676--3691},
  issn = {1053-587X, 1941-0476},
}

@article{sionGeneralMinimaxTheorems1958,
  title = {On General Minimax Theorems},
  author = {Sion, Maurice},
  year = {1958},
  month = mar,
  journal = {Pacific Journal of Mathematics},
  volume = {8},
  number = {1},
  pages = {171--176},
  issn = {0030-8730, 0030-8730},
  urldate = {2024-06-20},
  langid = {english}
}

@article{xuUnifiedSingleloopAlternating2023,
  title = {A Unified Single-Loop Alternating Gradient Projection Algorithm for Nonconvex--Concave and Convex--Nonconcave Minimax Problems},
  author = {Xu, Zi and Zhang, Huiling and Xu, Yang and Lan, Guanghui},
  year = {2023},
  month = sep,
  journal = {Mathematical Programming},
  volume = {201},
  number = {1-2},
  pages = {635--706},
  issn = {0025-5610, 1436-4646},
}

@article{xuRiemannianAlternatingDescent2024,
  title={A Riemannian Alternating Descent Ascent Algorithmic Framework for Nonconvex-Linear Minimax Problems on Riemannian Manifolds},
  author={Xu, Meng and Jiang, Bo and Liu, Ya-Feng and So, Anthony Man-Cho},
  journal={arXiv:2409.19588},
  year={2024}
}

@article{bernhardTheoremDanskinApplication1995,
  title = {On a Theorem of {{Danskin}} with an Application to a Theorem of {{Von Neumann-Sion}}},
  author = {Bernhard, Pierre and Rapaport, Alain},
  year = {1995},
  month = apr,
  journal = {Nonlinear Analysis: Theory, Methods \& Applications},
  volume = {24},
  number = {8},
  pages = {1163--1181},
  issn = {0362546X},
  doi = {10.1016/0362-546X(94)00186-L},
  urldate = {2024-06-20},
  copyright = {https://www.elsevier.com/tdm/userlicense/1.0/},
  langid = {english}
}

@InProceedings{linGradientDescentAscent,
  title = 	 {On Gradient Descent Ascent for Nonconvex-Concave Minimax Problems},
  author =       {Lin, Tianyi and Jin, Chi and Jordan, Michael},
  booktitle = 	 {Proceedings of the 37th International Conference on Machine Learning},
  pages ={6083-6093},
  year ={2020},
  editor = 	 {III, Hal Daumé and Singh, Aarti},
  volume = 	 {119},
  series = 	 {Proceedings of Machine Learning Research},
  month = 	 {13--18 Jul},
}

@InProceedings{linOptimalAlgorithmsMinimax,
  title = 	 {Near-Optimal Algorithms for Minimax Optimization},
  author =       {Lin, Tianyi and Jin, Chi and Jordan, Michael I.},
  booktitle = 	 {Proceedings of Thirty Third Conference on Learning Theory},
  pages = 	 {2738--2779},
  year = 	 {2020},
  editor = 	 {Abernethy, Jacob and Agarwal, Shivani},
  volume = 	 {125},
  series = 	 {Proceedings of Machine Learning Research},
  month = 	 {09--12 Jul},
}

@article{ostrovskiiEfficientSearchFirstOrder2021,
  title = {Efficient {{Search}} of {{First-Order Nash Equilibria}} in {{Nonconvex-Concave Smooth Min-Max Problems}}},
  author = {Ostrovskii, Dmitrii M. and Lowy, Andrew and Razaviyayn, Meisam},
  year = {2021},
  month = jan,
  journal = {SIAM Journal on Optimization},
  volume = {31},
  number = {4},
  pages = {2508--2538},
  issn = {1052-6234, 1095-7189},
}

@article{panEfficientAlgorithmNonconvexlinear2021a,
  title = {An Efficient Algorithm for Nonconvex-Linear Minimax Optimization Problem and Its Application in Solving Weighted Maximin Dispersion Problem},
  author = {Pan, Weiwei and Shen, Jingjing and Xu, Zi},
  year = {2021},
  month = jan,
  journal = {Computational Optimization and Applications},
  volume = {78},
  number = {1},
  pages = {287--306},
  issn = {0926-6003, 1573-2894},
}

@inproceedings{zhangSingleLoopSmoothedGradient,
 author = {Zhang, Jiawei and Xiao, Peijun and Sun, Ruoyu and Luo, Zhiquan},
 booktitle = {Advances in Neural Information Processing Systems},
 editor = {H. Larochelle and M. Ranzato and R. Hadsell and M.F. Balcan and H. Lin},
 pages = {7377--7389},
 title = {A Single-Loop Smoothed Gradient Descent-Ascent Algorithm for Nonconvex-Concave Min-Max Problems},
volume = {33},
 year = {2020}
}

@InProceedings{yangFasterSingleloopAlgorithms,
  title = 	 { Faster Single-loop Algorithms for Minimax Optimization without Strong Concavity },
  author =       {Yang, Junchi and Orvieto, Antonio and Lucchi, Aurelien and He, Niao},
  booktitle = 	 {Proceedings of The 25th International Conference on Artificial Intelligence and Statistics},
  pages = 	 {5485--5517},
  year = 	 {2022},
  editor = 	 {Camps-Valls, Gustau and Ruiz, Francisco J. R. and Valera, Isabel},
  volume = 	 {151},
  series = 	 {Proceedings of Machine Learning Research},
  month = 	 {28--30 Mar},
}

@article{heApproximationProximalGradient2024,
  title={An approximation proximal gradient algorithm for nonconvex-linear minimax problems with nonconvex nonsmooth terms},
  author={He, Jiefei and Zhang, Huiling and Xu, Zi},
  journal={Journal of Global Optimization},
  volume={90},
  number={1},
  pages={73--92},
  year={2024},
  publisher={Springer},
  langid = {english}
}

@article{2022zeroth,
  title={Zeroth-order single-loop algorithms for nonconvex-linear minimax problems},
  author={Shen, Jingjing and Wang, Ziqi and Xu, Zi},
  journal={Journal of Global Optimization},
  volume={87},
  number={2},
  pages={551--580},
  year={2023},
  publisher={Springer}
}

\end{sloppypar}
\end{document}